\documentclass[a4paper,leqno]{amsart}
\usepackage[margin=30truemm]{geometry}
\usepackage{fancyhdr,appendix,graphicx,layout}
\usepackage{stix2}
\usepackage{amsmath,amsthm}
\usepackage{enumitem}
\usepackage{physics}

\usepackage[numbers]{natbib}
\usepackage{bbm}
\usepackage{bm}
\usepackage{tikz}
\usepackage{tikz-cd}
\usetikzlibrary{positioning,arrows.meta,decorations.markings,calc}
\tikzset{
    cells={font=\everymath\expandafter{\the\everymath\displaystyle}},
}
\usepackage{caption}
\usepackage[pagewise]{lineno}

\usepackage{float}

\renewcommand{\emph}[1]{\textcolor{blue}{\textit{#1}}}

\usepackage{hyperref}
\hypersetup{
    colorlinks=true,
    linkcolor=red,  
    filecolor=blue,  
    urlcolor=orange, 
    citecolor=violet   
}
\usepackage{cleveref}

\makeatletter
\newcommand\ReDeclareMathOperator[2]{%
    \begingroup \escapechar\m@ne\xdef\@gtempa{{\string#1}}\endgroup
    \expandafter\@ifundefined\@gtempa
        {\@latex@error{Command \string#1 undefined}\@ehc}
        \relax
    \let\@ifdefinable\@rc@ifdefinable
    \DeclareMathOperator#1{#2}}
\makeatother

\crefname{thm}{Theorem}{Theorems}
\crefname{dfn}{Definition}{Definitions}
\crefname{dfnprop}{Definition-Proposition}{Definition-Proposition}
\crefname{prop}{Proposition}{Propositions}
\crefname{lem}{Lemma}{Lemmas}
\crefname{cor}{Corollary}{Corollaries}
\crefname{clm}{Claim}{Claims}
\crefname{ass}{Assumption}{Assumption}
\crefname{cond}{Condition}{Condition}
\crefname{nota}{Notation}{NOtation}
\crefname{conj}{Conjecture}{Conjecture}
\crefname{fct}{Fact}{Facts}
\crefname{rmk}{Remark}{Remarks}
\crefname{eg}{Example}{Examples}
\crefname{que}{Question}{Questions}
\crefname{figure}{Figure}{Figures}
\crefname{table}{Table}{Tables}
\crefname{section}{Section}{Sections}
\crefname{subsection}{Subsection}{Subsections}
\crefname{appendix}{Appendix}{Appendices}
\crefname{equation}{}{}
\crefname{main}{Theorem}{Theorems}

\theoremstyle{definition}
\newtheorem{thm}{Theorem}[section]

\newtheorem{dfn}[thm]{Definition}

\newtheorem{prop}[thm]{Proposition}
\newtheorem{lem}[thm]{Lemma}
\newtheorem{cor}[thm]{Corollary}

\newtheorem{rmk}[thm]{Remark}
\newtheorem{eg}[thm]{Example}

\numberwithin{equation}{section}
\makeatletter
\let\c@equation\c@thm
\makeatother

\newcommand{\A}{\mathcal A}
\newcommand{\D}{\mathcal D}
\renewcommand{\H}{\mathcal H}
\newcommand{\T}{\mathcal T}
\newcommand{\U}{\mathcal U}
\newcommand{\V}{\mathcal V}
\newcommand{\W}{\mathcal W}

\newcommand{\RR}{\mathbf{R}}
\newcommand{\LL}{\mathbf{L}}

\newcommand{\ZZ}{\mathbb{Z}}

\newcommand{\mono}{\hookrightarrow}

\newcommand{\xto}{\xrightarrow}

\newcommand{\To}{\Rightarrow}

\DeclareMathOperator{\add}{add}

\DeclareMathOperator{\rad}{rad}
\ReDeclareMathOperator{\top}{top}

\DeclareMathOperator{\ind}{ind}

\let\mod\relax
\DeclareMathOperator{\mod}{mod}

\DeclareMathOperator{\grmod}{grmod}
\DeclareMathOperator{\Grmod}{Grmod}
\DeclareMathOperator{\sdim}{\st{dim}}

\DeclareMathOperator{\thick}{thick}

\DeclareMathOperator{\per}{per}
\DeclareMathOperator{\pvd}{pvd}

\DeclareMathOperator{\Hom}{Hom}
\DeclareMathOperator{\End}{End}
\DeclareMathOperator{\Aut}{Aut}

\DeclareMathOperator{\REnd}{\RR End}
\DeclareMathOperator*{\ten}{\otimes}
\DeclareMathOperator{\Tot}{Tot}
\DeclareMathOperator{\Ext}{Ext}

\DeclareMathOperator{\NC}{NC}
\DeclareMathOperator{\wide}{wide}
\DeclareMathOperator{\Kr}{Kr}
\DeclareMathOperator{\cox}{cox}
\ReDeclareMathOperator{\l}{\ell}
\newcommand{\ora}{\overrightarrow}

\newcommand{\st}[1]{\underline{#1}}

\DeclareMathOperator{\la}{\langle}
\DeclareMathOperator{\ra}{\rangle}
\newcommand{\wt}{\widetilde}

\newcommand{\vv}{\mathsf{v}}
\newcommand{\ww}{\mathsf{w}}

\title{A classification of derived-discrete graded algebras}
\author{Riku Fushimi and Bohan Xing}
\date{\today}

\newcommand{\Addresses}{{
  \bigskip
  \footnotesize

  R. Fushimi, \textsc{Department of mathematics, Nagoya University, Chikusa-ku, Nagoya 464-8602, Japan}\par\nopagebreak
  \textit{E-mail address}: \texttt{fushimi.riku.h9@s.mail.nagoya-u.ac.jp}

  \medskip

  B. Xing, \textsc{School of Mathematical Sciences, Laboratory of Mathematics and Complex Systems, Beijing Normal University, Beijing 100875, P.R. China}\par\nopagebreak
\textit{E-mail address}: \texttt{bhxing@mail.bnu.edu.cn}
}}

\begin{document}

\begin{abstract}

A finite-dimensional algebra is derived-discrete, in the sense of Vossieck, precisely when it is a piecewise hereditary algebra of Dynkin type or a gentle one-cycle algebra not satisfying the clock condition. Building on the discrete triangulated categories of Broomhead, Pauksztello, and Ploog, we study derived-discreteness for locally finite non-positively graded algebras, regarded as connective locally finite dg algebras with trivial differential. Our main result extends Vossieck's classification to this setting: such a graded algebra is derived-discrete if and only if it is graded Morita equivalent to a piecewise hereditary algebra of Dynkin type, or it is a graded gentle one-cycle algebra not satisfying the graded clock condition. Along the way, using surface models we prove the conjecture of Kalck and Yang that the graded clock condition is invariant under derived equivalence. We also establish a restriction on semi-orthogonal decompositions of the bounded derived category of a path algebra of Dynkin type $D$.
\end{abstract}

\maketitle

\tableofcontents

\section{Introduction}

The perfectly valued derived category $\pvd A:=\{X\in\D(A)\mid \dim H^*X<\infty\}$ is a central derived invariant of a locally finite dg algebra $A$, though its structure is often difficult to describe explicitly. Vossieck \cite{Vos01} singled out the class of \emph{derived-discrete algebras}, whose bounded derived categories are combinatorially controlled: for every prescribed cohomology dimension vector there are only finitely many indecomposable objects of $\D^b(\mod \Lambda)=\pvd \Lambda$ realizing it. He obtained a remarkably clean classification.

\begin{thm}[\cite{Vos01}; see \cref{thm:Vossieck}]\label{thm:intro-Vossieck}
	A finite-dimensional algebra $\Lambda$ is derived-discrete if and only if it is either
	\begin{enumerate}
		\item a piecewise hereditary algebra of Dynkin type, or
		\item a gentle one-cycle algebra not satisfying the clock condition.
	\end{enumerate}
\end{thm}

The derived categories arising in case (2) were described explicitly by Bobi\'{n}ski, Gei\ss, and Skowro\'{n}ski \cite{BGS04}. This work was later placed in a broader categorical framework by Broomhead, Pauksztello, and Ploog \cite{BPP18}, who introduced the notion of \emph{discrete triangulated categories}. Within this framework, Adachi, Mizuno, and Yang \cite{AMY19} studied discreteness properties of silting objects and $t$-structures. One advantage of this categorical viewpoint is that it applies well beyond the setting of trivially graded algebras.

In this paper, we work with locally finite non-positively graded algebras. Such an algebra $A$ may be regarded as a connective, locally finite dg algebra with trivial differential, and we call $A$ \emph{derived-discrete} when its perfectly valued derived category $\pvd A$ is discrete in the sense of \cite{BPP18} (see Definition \ref{dfn:dd}); for trivially graded $A$ this recovers Vossieck's notion. The bridge between the derived category of graded modules and $\pvd A$ is supplied by Kalck-Yang \cite{KY18}. Our main result extends \cref{thm:intro-Vossieck} to all locally finite non-positively graded algebras.

\begin{thm}[see \cref{thm:main}]\label{thm:intro-main}
	A locally finite non-positively graded algebra $A = kQ/I$ is derived-discrete if and only if it is either
	\begin{enumerate}
		\item graded Morita equivalent to a piecewise hereditary algebra of Dynkin type, or
		\item a graded gentle one-cycle algebra not satisfying the graded clock condition.
	\end{enumerate}
\end{thm}

Case (2) is characterized by the graded clock condition, a degree-sensitive refinement of Vossieck's clock condition due to Kalck and Yang \cite{KY18}. Let $d_+$ (resp. $d_-$) denote the difference between the number of clockwise (resp. counterclockwise) oriented relations and the sum of the degrees of all clockwise (resp. counterclockwise) oriented arrows. Then the graded gentle one-cycle algebra $A$ satisfies the \emph{graded clock condition} if $d_+ = d_-$. For the trivial grading, this condition reduces to the equality of the numbers of clockwise and counterclockwise oriented relations. 

Kalck and Yang conjectured \cite[Conjecture 9.5]{KY18} that the graded clock condition is invariant under derived equivalence. Following the geometric approach developed in \cite{HKK17,LP20}, we confirm their conjecture. Moreover, we obtain the following results.

\begin{thm}[see Proposition \ref{prop:gr-clock-cdt} and \cref{thm:der-dis-gentle}]\label{prop:intro-clock}
	Let $A = kQ/I$ be a graded gentle one-cycle algebra. Then the following conditions are equivalent:
\begin{enumerate}
\item $A$ satisfies the graded clock condition;
\item $A$ is derived equivalent either to the path algebra $kQ$ with the trivial grading, where $Q$ is the quiver of type $\widetilde{A}$, or to an algebra $kQ/I$, where $Q$ is an $n$-cycle, $I$ is generated by all paths of length two, and exactly one arrow has positive degree $n$.
\end{enumerate}
Moreover, if $A$ is locally finite and non-positively graded, then the following conditions are equivalent:
\begin{enumerate}
\item $A$ is derived-discrete;
\item $A$ is silting-discrete;
\item $A$ is not derived equivalent to the path algebra of type $\widetilde{A}$ with the trivial grading;
\item $A$ does not satisfy the graded clock condition.
\end{enumerate}
\end{thm}


To deduce \cref{thm:intro-main} from the gentle one-cycle case, we consider the finite-dimensional algebras $\wt{A}^{[-n,0]}$ with $\mod \wt{A}^{[-n,0]} \simeq \grmod_0^{[-n,0]} A$, where $\grmod_0^{[-n,0]} A$ denotes the category of finite-dimensional graded $A$-modules concentrated in degrees $[-n,0]$. If $A$ is derived-discrete, then so is each $\widetilde{A}^{[-n,0]}$ (see Proposition~\ref{prop:DD-DD}), and hence each of these truncations is covered by \cref{thm:intro-Vossieck}. A degeneration argument then transfers information from these truncations back to $A$ and yields strong restrictions on $A$ itself. Along this way, the main remaining difficulty is to exclude the Dynkin type $D$ case. For this purpose, we prove the following result on semi-orthogonal decompositions, which may be of independent interest.

\begin{thm}[see \cref{thm:SOD}]\label{thm:intro-SOD}
	In any semi-orthogonal decomposition $\D^b(\mod kD_n) = \T_1 \bot \T_2$, at least one of $\T_1, \T_2$ has no direct summand equivalent to $\D^b(\mod kD_l)$ for any $l \ge 4$.
\end{thm}

The proof of Theorem~\ref{thm:intro-SOD} uses the theory of noncrossing partitions of type $D$~\cite{AR04} and the Ingalls--Thomas bijection between wide subcategories and noncrossing partitions~\cite{IT09}: the key point is that if one factor of a semi-orthogonal decomposition contains a $D$-type component of size $\geq 4$, then the corresponding noncrossing partition has a zero block, and a combinatorial lemma on Kreweras complements forces the other factor to contain no such component.

\bigskip
\noindent
{\bf Organization.}
In Section \ref{sec:2}, we recall some basic definitions and properties of gentle algebras, derived categories, and derived-discrete dg algebras. We also associate to each non-positively graded algebra $A$ a family of algebras $\widetilde{A}^{[-n,0]}$ obtained by truncating the grading of $A$. In Section \ref{sec:4}, we use the surface model to classify graded gentle one-cycle algebras up to derived equivalence and to characterize which of them are derived-discrete. In Section \ref{sec:3}, we use semi-orthogonal decompositions to characterize the bounded derived category of the path algebra of Dynkin type $D$. Finally, in Section \ref{sec:5}, we classify non-positively graded derived-discrete algebras.

\medskip

\noindent
{\bf Conventions and notation.}
For a path $p$ in a quiver, we denote by $s(p)$ its source vertex and by $t(p)$ its target vertex. We compose arrows from right to left. Unless otherwise stated, all graded algebras are locally finite over an algebraically closed field $k$, and all modules are right modules.

For a triangulated category $\T$ and its subset $\U\subseteq\T$, we put
\begin{itemize}
    \item $\add \U$: the minimal additive subcategory containing $\U$ and closed under direct summands,
    \item $\thick \U$: the minimal thick subcategory of $\T$ containing $\U$,
    \item $\U^\bot:=\{X\in\T\mid\Hom_\T(U,X)=0\text{ for every }U\in\U\}$,
    \item $\U^{\bot_{0,1}}:=\{X\in\T\mid\Hom_\T(U,X)=0=\Hom_\T(U,\Sigma X)\text{ for every }U\in\U\}$.
\end{itemize}
For two subsets $\U,\V\subseteq\T$, we put
\[
\U\ast\V:=\{X\in\T\mid\text{ there is an exact triangle } U\to X\to V\to\text{ such that }U\in\U\text{ and }V\in\V\}
\]

\bigskip
\noindent
\textbf{Acknowledgements.}
The first-named author would like to express his sincere gratitude to his supervisor Akira Ishii for his continuous support and encouragement. He also thanks Nao Mochizuki and Ryu Tomonaga for helpful conversations. The second-named author is deeply grateful to Aaron Chan for valuable discussions and for his support during the author's stay at Nagoya University. He would also like to thank Zhengfang Wang for generously sharing his knowledge of gentle algebras. The second-named author was supported by the China Scholarship Council (No. 202506040127).

\section{Preliminaries}\label{sec:2}

\subsection{Gentle one-cycle algebras and the clock condition}

Gentle algebras were introduced in~\cite{AS87} and form a classical class of algebras. In recent years, their study has seen renewed interest, particularly through surface models; see, for example,~\cite{OPS18,BS21}.

We now recall the definition of a gentle algebra.

\begin{dfn}
    Let $Q$ be a finite quiver and let $I$ be an ideal of $kQ$. The algebra $\Lambda=kQ/I$ is called a \emph{gentle algebra} if the following conditions hold:
\begin{enumerate}
    \item for each vertex $v\in Q_0$, there are at most two arrows starting at $v$ and at most two arrows ending at $v$;
    \item for each arrow $\alpha\in Q_1$, there is at most one arrow $\beta\in Q_1$ such that $\alpha\beta\notin I$, and at most one arrow $\gamma\in Q_1$ such that $\gamma\alpha\notin I$;
    \item for each arrow $\alpha\in Q_1$, there is at most one arrow $\beta\in Q_1$ such that $\alpha\beta\in I$, and at most one arrow $\gamma\in Q_1$ such that $\gamma\alpha\in I$;
    \item the ideal $I$ is generated by paths of length two.
\end{enumerate}
\end{dfn}

We call a gentle algebra $\Lambda=kQ/I$ a \emph{gentle one-cycle algebra} if the underlying graph of $Q$ contains exactly one cycle, and a graded algebra $\Lambda=kQ/I$ \emph{graded gentle} if its underlying ungraded algebra is gentle.

Now let $\Lambda=kQ/I$ be a gentle one-cycle algebra. We say that $\Lambda$ satisfies the \emph{clock condition} if, along the unique cycle in the underlying graph of $Q$, the number of relations oriented clockwise equals the number of relations oriented counterclockwise.

Note that finite-dimensional gentle one-cycle algebras which do not satisfy the clock condition provide a fundamental source of derived-discrete algebras; see Subsection~\ref{subsec:DD} for the relevant definitions. They have been studied in depth, including their classification and the structure of their derived categories; see, for example,~\cite{Vos01,BGS04}.

\subsection{Derived categories}

For a dg algebra $A$, we denote by $\D(A)$ the derived category of dg $A$-modules; we refer the reader to \cite{K94} for the definition and basic properties. We say that two dg algebras $A$ and $B$ are \emph{derived equivalent} if $\D(A)$ and $\D(B)$ are equivalent as triangulated categories. 

\begin{rmk}\label{rmk:derived equivalence}
    For a graded algebra $A$, we denote by $\D(A)$ the derived category of $A$ viewed as a dg algebra with trivial differential. Throughout this paper, two graded algebras $A$ and $B$ are said to be derived equivalent if they are derived equivalent as dg algebras.
\end{rmk}

\begin{dfn}
	Let $A$ be a dg algebra.
	\begin{itemize}
		\item The \emph{perfect derived category} of $A$ is $\per A:=\thick_{\D(A)} A$.
        \item The \emph{perfectly valued derived category} of $A$ is $\pvd A:=\{X\in\D(A)\mid \dim H^*X<\infty\}$.
        \item For a subset $I\subseteq\ZZ$, $\pvd^IA$ denotes the full subcategory of $\pvd A$ consisting of objects $M$ with $H^j(M)=0$ for all $j\notin I$.
		\item $A$ is \emph{connective} if $H^{>0}(A) = 0$.
		\item $A$ is \emph{locally finite} if $\dim H^i(A) < \infty$ for every $i \in \mathbb{Z}$.
	\end{itemize}
\end{dfn}

\begin{dfn}
	Let $A$ be a graded algebra.
	\begin{itemize}
		\item $\Grmod A$ (resp. $\grmod_0 A$) denotes the abelian category of graded $A$-modules (resp. finite-dimensional graded $A$-modules).
		\item For $M \in \Grmod A$ and $l \in \mathbb{Z}$, the \emph{grading shift} $M(l)$ is defined by $M(l)_i:=M_{i+l}$.
		\item For a subset $I\subseteq\ZZ$, $\Grmod^I A$ denotes the full subcategory of $\Grmod A$ consisting of modules $M$ with $M_j = 0$ for all $j \notin I$, and $\grmod_0^IA:=\Grmod^IA\cap\grmod_0 A$.
		\item We write $\Sigma(-1)$ for the autoequivalence $M \mapsto (\Sigma M)(-1)$ of $\D(\Grmod A)$, where $\Sigma$ denotes the cohomological shift.
	\end{itemize}
\end{dfn}

\begin{rmk}\label{rmk:t-strs}
    For a graded algebra $A$, we have a standard $t$-structure $\bigl(\D^{\le 0}(\Grmod A),\, \D^{\ge 0}(\Grmod A)\bigr)$, whose heart is $\Grmod A$.
    
    If $A$ is non-positively graded, then we have $\D(\Grmod^{\le0}A)^\bot=\D(\Grmod^{>0}A)$.
\end{rmk}

A locally-finite graded algebra $A = kQ/I$ may be regarded as a connective locally finite dg algebra with trivial differential. The following theorem relates 
\[
\D^b_{\grmod_0A}(\Grmod A):=\{X\in\D(\Grmod A)\mid \dim H^*X<\infty\}
\]
to $\pvd A$, the perfectly valued derived category of $A$ viewed as a dg algebra. This bridge will be used in the proof of Proposition \ref{prop:DD-DD}.

\begin{thm}\label{thm:KY}
	Let $A$ be a locally finite non-positively graded algebra. The functor taking total complexes induces a fully faithful functor
	\[
		\Tot \colon \D^b_{\grmod_0 A}(\Grmod A)/\Sigma(-1) \mono \pvd A,
	\]
	where $\D^b_{\grmod_0 A}(\Grmod A)/\Sigma(-1)$ is the orbit category of $\D^b_{\grmod_0 A}(\Grmod A)$ with respect to $\Sigma(-1)$, whose objects are those of $\D^b_{\grmod_0 A}(\Grmod A)$ and whose morphism spaces are given by
	\[
		\Hom_{\D^b_{\grmod_0 A}(\Grmod A)/\Sigma(-1)}(X, Y) := \bigoplus_{i \in \ZZ} \Hom_{\D^b_{\grmod_0 A}(\Grmod A)}(X,\Sigma^i Y(-i)).
	\]
    If $A$ is graded hereditary, then this embedding is essentially surjective.
\end{thm}
\begin{proof}
    \cite[Theorem 5.1 (b)]{KY18} states that 
    \begin{align*}
        \Hom_{\D(\Grmod A)}\qty(X,\bigoplus_{i\in\ZZ}\Sigma^iY(-i))\simeq\Hom_{\D(A)}(\Tot X,\Tot Y)
    \end{align*}
    for every $X,Y\in\D(\Grmod A)$. If $X,Y\in\D^b_{\grmod_0 A}(\Grmod A)$, then Remark \ref{rmk:t-strs} shows that there is $n\ge0$ such that 
    \[
    \Hom_{\D(\Grmod A)}\qty(X,\bigoplus_{i\notin[-n,n]}\Sigma^iY(-i))=0.
    \]
    Thus, we have the desired isomorphism:
    \begin{align*}
        \bigoplus_{i \in \ZZ} \Hom_{\D(\Grmod A)}(X,\Sigma^i Y(-i))
        &=
        \bigoplus_{i \in [-n,n]} \Hom_{\D(\Grmod A)}(X,\Sigma^i Y(-i)) \\
        &=
        \Hom_{\D(\Grmod A)}\left(X,\bigoplus_{i \in [-n,n]} \Sigma^i Y(-i)\right) \\
        &=
        \Hom_{\D(\Grmod A)}\left(X,\bigoplus_{i \in \ZZ} \Sigma^i Y(-i)\right) \\
        &\simeq
        \Hom_{\D(A)}(\Tot X,\Tot Y).
    \end{align*}
    The last assertion follows from \cite[Theorem 5.1 (e)]{KY18}.
\end{proof}

\begin{rmk}
    $\Tot$ induces an embedding $\Tot\colon\ind\left(\D^b_{\grmod_0 A}(\Grmod A)\right)/\Sigma(-1)\mono\ind\left(\pvd A\right)$ by the argument in \cite[Section 3.1]{KY18}.
\end{rmk}

\subsection{Derived-discrete dg algebras}\label{subsec:DD} 

In this subsection, we define derived-discrete dg algebras and establish some of their basic properties.

\begin{dfn}\cite{BPP18}\label{dfn:derived-discrete}
    Let $\H$ be a bounded heart of $\T$ and $\sigma^i$ the associated cohomology functors. For an object $X$ of $\T$, define a function $v_X \colon \ZZ \to K_0(\H)$ by $v_X(i) = [\sigma^i(X)]$ for $i \in \ZZ$. We say that $\T$ is \emph{$\H$-discrete} if for any function $v \colon \ZZ \to K_0(\H)$, we have
    \begin{align*}
        \#\{X\in\ind \T\mid v_X=v\}<\infty.
    \end{align*}
\end{dfn}

\begin{thm}\cite[Theorem 7,9]{AMY19}\label{thm:H-discrete}
    Let $\H_1$ and $\H_2$ be two length hearts of $\T$. Assume that the number of isomorphism classes of simple objects in $\H_1$ (and hence in $\H_2$) is finite. Then $\T$ is $\H_1$-discrete if and only if it is $\H_2$-discrete.
\end{thm}

For a locally finite connective dg algebra $A$, the heart of the standard $t$-structure $(\pvd^{\le0}A,\pvd^{\ge0}A)$ can be identified with $\mod H^0(A)$.

\begin{dfn}\label{dfn:dd}
    Let $A$ be a locally finite connective dg algebra. We say that $A$ is \emph{derived-discrete} if $\pvd A$ is $(\mod H^0(A))$-discrete.
\end{dfn}

Note that when $A$ is a finite-dimensional algebra with the trivial grading, the above definition of derived-discreteness coincides with the classical definition of derived-discreteness in~\cite{Vos01}.

\begin{rmk}\label{rmk:dd-pre-under-der}
    By \cref{thm:H-discrete}, derived-discreteness is preserved under derived equivalences.
\end{rmk}

Vossieck's classification of derived-discrete algebras with the trivial grading can be stated as follows. Recall that an algebra $\Lambda$ is called \emph{piecewise hereditary} if it is derived equivalent to a hereditary algebra.

\begin{thm}\cite[Theorem 2.1]{Vos01}\label{thm:Vossieck}
    Let $\Lambda$ be a finite-dimensional algebra. Then the following conditions are equivalent:
    \begin{itemize}
        \item [(1)] $\Lambda$ is derived-discrete,
        \item [(2)] $\Lambda$ is either
        \begin{itemize}
            \item [(a)] a piecewise hereditary algebra of Dynkin type, 
            \item [(b)] a gentle one-cycle algebra not satisfying the clock condition.
        \end{itemize}
    \end{itemize}
\end{thm}
The main result of this paper, \cref{thm:main}, establishes a graded analogue of \cref{thm:Vossieck}, classifying derived-discrete locally finite graded algebras in terms of their graded structure.

The following result shows that derived-discreteness is inherited by idempotent subalgebras.

\begin{prop}\label{prop:idempotent}
    Let $A$ be a derived-discrete locally finite connective dg algebra. For every idempotent $e\in H^0(A)$, the dg algebra $eAe:=\REnd_A(eA)$ is also derived-discrete, where $eA\in\add_{\D(A)} A$ corresponds to the idempotent $e$.
\end{prop}
\begin{proof}
	We assume that $eAe$ is not derived-discrete, and deduce that $A$ is not derived-discrete. Consider the functor $F:=-\ten^\LL_{eAe}eA\colon\D(eAe)\to\D(A)$. By \cite[1.1]{Vos01}, for every $X\in\pvd(eAe)$, we have
	\[
		\sdim H^kF(X) \le \sum_{n\in\ZZ}\sum_j\dim H^nX(j)\cdot\sdim H^kF(\Sigma^{-n}S_j).
	\]
    Note that the right-hand side is a finite sum since $X \in \pvd(eAe)$.
	Since $eAe$ is not derived-discrete, there exist an integer $l>0$, an infinite index set $I$, and a family $\{X_i\}_{i\in I}\subseteq\pvd^{[-l,0]}(eAe)$ of pairwise non-isomorphic objects such that the dimension vector $\sdim H^nX_i\in K_0(\mod H^0(eAe))$ is independent of $i$ for all $n\in[-l,0]$.

	By the above argument and a pigeonhole argument, there exists an infinite subset $I'\subseteq I$ such that $\sdim H^kF(X_i)$ is independent of $i\in I'$ for all $k\in[-l,0]$.

	Since $A$ is connective and locally finite, the truncation $\sigma^{\ge-l}F(X_i)$ has cohomology concentrated in $[-l,0]$ with each cohomology group finite-dimensional, so $\sigma^{\ge-l}F(X_i)\in\pvd(A)$. Since the exact functor $(-)e\colon\D(A)\to\D(eAe)$ commutes with truncations and $F(X_i)e\simeq X_i\in\pvd^{[-l,0]}(eAe)$, we have
	\[
		\bigl(\sigma^{\ge-l}F(X_i)\bigr)e \simeq \sigma^{\ge-l}\bigl(F(X_i)e\bigr) \simeq F(X_i)e \simeq X_i.
	\]
	Hence the family $\{\sigma^{\ge-l}F(X_i)\}_{i\in I'}\subseteq\pvd(A)$ is pairwise non-isomorphic. Moreover, $\sdim H^k(\sigma^{\ge-l}F(X_i))=\sdim H^kF(X_i)$ for $k\in[-l,0]$ is independent of $i\in I'$. It follows that $A$ is not derived-discrete.
\end{proof}

\subsection{Algebras associated with a graded algebra}
In this subsection, we fix a locally finite non-positively graded algebra $A$ and associate to it a family of ordinary finite-dimensional algebras, which will be used in the proof of \cref{thm:main}.

For each $n\ge0$, the module $\bigoplus_{l=0}^nA(l)_{\ge-n}$ is a progenerator of $\grmod_0^{[-n,0]}A$, so the following definition yields an algebra controlling this category.

\begin{dfn}
	For $n\ge0$, define $\wt{A}^{[-n,0]}:=\End_{\Grmod A}\qty(\bigoplus_{l=0}^nA(l)_{\ge-n})$. This is an ordinary finite-dimensional algebra.
\end{dfn}

We have an equivalence
\[
	\Hom_{\Grmod A}\qty(\bigoplus_{l=0}^nA(l)_{\ge-n},-)\colon\grmod_0^{[-n,0]}A\simeq\mod\wt{A}^{[-n,0]}.
\]

\begin{lem}
    Let $n\ge0$.
    The inclusion $\grmod_0^{[-n,0]}\mono\Grmod A$ induces a fully faithful functor 
    \[
    \D^b(\grmod_0^{[-n,0]}A)\mono\D^b_{\grmod_0A}(\Grmod A).
    \]
\end{lem}
\begin{proof}
    It suffices to show that for $M,N\in\grmod_0^{[-n,0]}A$, the natural map $\Ext^i_{\grmod_0^{[-n,0]}A}(M,N)\to\Ext^i_{\Grmod A}(M,N)$ is an isomorphism for all $i\ge0$. Since $A$ is locally finite and non-positively graded, $M$ admits a projective resolution
    \[
    \cdots P^{-2}\to P^{-1}\to P^0\to M\to 0
    \]
    in $\Grmod A$ with $P^i\simeq\bigoplus_{k\ge0}A(k)^{n_{k,i}}$ for some $n_{k,i}\ge0$. Truncating each term, we obtain a projective resolution
    \[
    \cdots (P^{-2})_{\ge-n}\to (P^{-1})_{\ge-n}\to (P^0)_{\ge-n}\to M\to 0
    \]
    in $\grmod_0^{[-n,0]}A$. Since we have an isomorphism $\Hom_{\grmod_0^{[-n,0]}A}((P^{-i})_{\ge-n},N)\simeq\Hom_{\Grmod A}(P^{-i},N)$ for every $i\ge0$, we obtain the desired isomorphism
    \[
    \Ext^i_{\grmod_0^{[-n,0]}A}(M,N)\to\Ext^i_{\Grmod A}(M,N)
    \]
    for every $i\ge0$.
\end{proof}

The following proposition will be used to prove \cref{thm:main}
\begin{prop}\label{prop:DD-DD}
    If $A$ is derived-discrete, then $\wt{A}^{[-n,0]}$ is a derived-discrete algebra for every $n\ge0$.
\end{prop}
\begin{proof}
	We assume that $\wt{A}^{[-n,0]}$ is not derived-discrete for some $n\ge0$, and deduce that $A$ is not derived-discrete. Then, there exist an infinite index set $I$ and a family $\{M_i\}_{i\in I}\subseteq\D^b(\mod\wt{A}^{[-n,0]})$ of pairwise non-isomorphic objects such that the dimension vector $\sdim H^k(M_i)$ is independent of $i$ for every $k\in\ZZ$. Let $p_n\colon\D^b(\mod\wt{A}^{[-n,0]})\to\D^b_{\grmod_0A}(\Grmod A)/\Sigma(-1)$ be the composition
	\begin{align*}
		\D^b(\mod\wt{A}^{[-n,0]})\simeq\D^b(\grmod_0^{[-n,0]}A)\mono\D^b_{\grmod_0A}(\Grmod A)\to\D^b_{\grmod_0A}(\Grmod A)/\Sigma(-1).
	\end{align*}
	Since $p_n(M_i)\simeq p_n(M_j)$ implies $M_i\simeq\Sigma^{-l}M_j(l)$ for some $l\in\ZZ$, and since $\Sigma^{-l}M_j(l)\notin\D^b(\grmod_0^{[-n,0]}A)$ for $l\notin[0,n]$, each isomorphism class in the image of $p_n$ contains at most $n+1$ elements of the family. Since $I$ is infinite, the family $\{p_n(M_i)\}_{i\in I}$ has infinitely many isomorphism classes. 

    By \cref{thm:KY}, we have a fully-faithful functor
    \begin{align*}
        \Tot\colon\D^b_{\grmod_0A}(\Grmod A)/\Sigma(-1)\mono\pvd A.
    \end{align*}
    Applying \cite[1.1]{Vos01} to $\Tot\circ p_n$, we deduce that $A$ is not derived-discrete.
\end{proof}

For a locally finite non-positively graded algebra $A\simeq kQ/I$, we give an explicit quiver with relations presentation of $\wt{kQ/I}^{[-n,0]}$.

\begin{rmk}
	Let $A$ be a locally finite non-positively graded algebra. If $A_0/\rad A^0=k^n$, then $A\simeq kQ/I$ for some locally finite non-positively graded quiver $Q$ and a graded ideal $I\subseteq kQ_{\ge2}$, analogously to the finite-dimensional case.
\end{rmk}

\begin{dfn}
    Let $n\ge 0$.
    We define a quiver $\wt{Q}^{[-n,0]}$ by
    \begin{itemize}
        \item $\wt{Q}^{[-n,0]}_0=\{(i,l)\mid i\in Q_0,\ l\in[-n,0]\}$,
        \item $\wt{Q}^{[-n,0]}_1=\{(i,l)\xto{\alpha_l}(j,l+\deg\alpha)\mid i\xto{\alpha}j,\ \{l,l+\deg\alpha\}\subseteq[-n,0]\}$
    \end{itemize}
\end{dfn}

It is straightforward to see that $\wt{kQ}^{[-n,0]}\simeq k\wt{Q}^{[-n,0]}$.

\begin{eg}
    If $Q=$
    \begin{tikzcd}
        1\ar[r,yshift=0.7ex]\ar[r,yshift=-0.7ex,"-1"'] & 2\ar[r] & 3
    \end{tikzcd},
    then
    $\wt{Q}^{[-2,0]}=$
    \begin{tikzcd}[row sep=0.5em]
        (1,0)\ar[r]\ar[rd] & (2,0)\ar[r] & (3,0) \\
        (1,-1)\ar[r]\ar[rd] & (2,-1)\ar[r] & (3,-1) \\
        (1,-2)\ar[r] & (2,-2)\ar[r] & (3,-2) 
    \end{tikzcd}. 
\end{eg}

By using $\wt{Q}^{[-n,0]}$, we can give the following example of a derived-discrete locally finite non-positively graded algebra.
\begin{eg}\label{eg:n-cycle}
    For a graded $n$-cycle $Q$ with a unique negative arrow, the quiver $\wt{Q}^{[-n,0]}$ is a disjoint union of Dynkin quivers of type $A$. By \cref{thm:KY}, we have
    \[
    \#\ind\qty(\pvd^{[-n,0]}kQ)=\#\ind\qty(\grmod_0^{[-n,0]}kQ)=\#\ind\qty(\mod k\wt{Q}^{[-n,0]})<\infty.
    \]
    In particular, $kQ$ is derived-discrete.
\end{eg}

\begin{dfn}
    For a graded ideal $I\subseteq kQ$, we define an ideal $\wt{I}^{[-n,0]}\subseteq k\wt{Q}^{[-n,0]}$ in obvious way. Then we have an isomorphism $\wt{kQ/I}^{[-n,0]}\simeq k\wt{Q}^{[-n,0]}/\wt{I}^{[-n,0]}$.
\end{dfn}

\begin{eg}
    If $Q=$
    \begin{tikzcd}
        1\ar[r,yshift=0.7ex,"\alpha"]\ar[r,yshift=-0.7ex,"\beta"'] & 2\ar[r,"\gamma"] & 3
    \end{tikzcd} with $\deg(\alpha)=0=\deg(\gamma)$ and $\deg(\beta)=-1$, and $I=\langle\gamma\beta\rangle$,
    then $\wt{I}^{[-2,0]}=\langle\gamma_{-1}\beta_0,\gamma_{-2}\beta_{-1}\rangle$.
\end{eg}

\section{Graded gentle one-cycle algebras}\label{sec:4}

    In this section, we discuss when a graded gentle one-cycle algebra is derived-discrete. By Remark \ref{rmk:dd-pre-under-der}, this property is invariant under the derived equivalence. Therefore, we first discuss the classification of gentle one-cycle algebras up to derived equivalence.

    The derived equivalence classifications of gentle one-cycle algebras are known algebraically in the graded case \cite{KY18} and via surface models in the ungraded case \cite{LZ21}. In order to adapt these results to our non-positively graded setting, we formulate the graded classification in terms of surface models.

    We note that in the following three subsections, we discuss general graded gentle one-cycle algebras, which need not be locally finite or non-positively graded.

\subsection{Marked ribbon graphs and marked annuli}

    We first recall the definition of a marked ribbon graph in \cite{OPS18}.

\begin{dfn}\label{dfn:marked-ribbon-graph}
		A \emph{marked ribbon graph} is a tuple $\Gamma=(V,H,s,\iota,\rho,m)$, where
		\begin{enumerate}
			\item $V$ (also denoted by $V(\Gamma)$) is a finite set whose elements are called vertices;
			
			\item $H$ (also denoted by $H(\Gamma)$) is a finite set whose elements are called half-edges;
			
			\item $s\colon H\rightarrow V$ is a function;
			
			\item $\iota\colon H\rightarrow H$ is an involution without fixed points;
			
			\item $\rho\colon H\rightarrow H$ is a permutation whose cycles correspond to the sets $s^{-1}(\vv)$, $\vv\in V$;

            \item A marking is a map $m\colon V'\to H$ defined on a subset $V'\subseteq V$, which assigns to each vertex $\vv\in V'$ a half-edge $m(\vv)$ incident with $\vv$, that is, satisfying $s(m(\vv))=\vv$.
			
		\end{enumerate}
		The $\iota$-orbits are called the edges of $\Gamma$. In particular, if we set
\[
E(\Gamma) := H/\langle \iota \rangle,
\]
then $(V(\Gamma), E(\Gamma))$ is the underlying combinatorial graph of $\Gamma$.
	\end{dfn}

    Indeed, by~\cite{Sch15,OPS18}, every gentle algebra, and hence every gentle one-cycle algebra, can be associated with a marked ribbon graph as follows.

    Let $A=kQ/I$ be a gentle algebra. Since $A$ is monomial and special biserial, there is a unique set $\mathcal{M}$ consisting of all primitive cycles $p$, up to cyclic permutation, such that $p^n\neq 0$ in $A$ for all $n\geq 1$, and all maximal paths in $A$. Here a cycle $p$ is called \emph{primitive} if there is no cycle $q$ and no integer $m\geq 2$ such that $p=q^m$, and a path $p$ is called \emph{maximal} if, for every arrow $\alpha\in Q_1$, one has $\alpha p=0=p\alpha$ in $A$.

    We associate $A$ with a marked ribbon graph $\Gamma_A:=(V,H,s,\iota,\rho,m)$ as follows.

\begin{enumerate}
\item The set of vertices is $V:=\overline{\mathcal{M}}$,
where
$$
\overline{\mathcal{M}} := \mathcal{M} \cup \left\{\, i \in Q_0 \ \middle|\ 
\begin{array}{l}
i \text{ is a source with a single arrow starting at } i, \\
\text{or } i \text{ is a sink with a single arrow ending at } i, \\
\text{or there is a single arrow } \alpha \text{ ending at } i \\
\text{and a single arrow } \beta \text{ starting at } i \text{ with } \beta\alpha \notin I
\end{array}
\right\},
$$
Thus, the vertices of $\Gamma_A$ correspond to the elements of $\mathcal{M}$, together with the above extra trivial paths.

\item The set of half-edges $H=H(\Gamma_A)$ is the set of all occurrences of vertices of $Q$ in the paths $p_{\vv}\in \overline{\mathcal{M}}$.
More explicitly, if
$$
p_{\vv}=i_0\xrightarrow{\alpha_1} i_1\xrightarrow{\alpha_2}\cdots
\xrightarrow{\alpha_\ell} i_\ell,
$$
then $p_{\vv}$ contributes the half-edges
$$
h_{\vv,0},h_{\vv,1},\ldots,h_{\vv,\ell},
$$
where $h_{\vv,j}$ corresponds to the occurrence of the vertex $i_j$ at the $j$-th position of the path $p_{\vv}$. Thus, even if $i_j=i_{j'}$ for $j\neq j'$, the half-edges $h_{\vv,j}$ and $h_{\vv,j'}$ are considered distinct.

If $\vv$ is a trivial path at a vertex $i\in Q_0$, then $p_{\vv}=i$ contributes a unique half-edge, denoted by $h_{\vv,0}$.

\item The map $s\colon H\to V$ is given by
$$
s(h_{\vv,j})=\vv.
$$

\item The involution $\iota\colon H\to H$
is defined as follows. For each $i\in Q_0$, let
$h_{\vv,j},\ h_{\ww,k}$
be the two half-edges corresponding to the two occurrences of $i$ among the paths $p\in\overline{\mathcal{M}}$. Then we set
$$
\iota(h_{\vv,j})=h_{\ww,k},
\qquad
\iota(h_{\ww,k})=h_{\vv,j}.
$$

\item The permutation $\rho\colon H\to H$ is defined by the order of occurrences along each path. More precisely, if
$$
p_{\vv}=i_0\xrightarrow{\alpha_1} i_1\xrightarrow{\alpha_2}\cdots
\xrightarrow{\alpha_\ell} i_\ell,
$$
we set
$$
\rho(h_{\vv,j})=h_{\vv,j+1}
\quad\text{for }0\leq j<\ell,
$$
and
$$
\rho(h_{\vv,\ell})=h_{\vv,0}.
$$
If $\vv$ is a trivial path, then we set $\rho(h_{\vv,0})=h_{\vv,0}.$

\item Finally, we define the marking $m\colon V'\to H$, where $V'\subseteq V$ consists of the vertices which do not correspond to primitive cycles, as follows. If $\vv\in V'$ corresponds to a non-trivial path
$$
p_{\vv}=i_0\xrightarrow{\alpha_1} i_1\xrightarrow{\alpha_2}\cdots
\xrightarrow{\alpha_\ell} i_\ell,
$$
then we set
$$
m(\vv):=h_{\vv,\ell}.
$$
If $\vv\in V'$ corresponds to a trivial path, then we set
$$
m(\vv):=h_{\vv,0}.
$$
\end{enumerate}

    We illustrate the above construction with the following example.

\begin{eg}\label{eg:easy-example}
    Consider the following quiver $Q$:
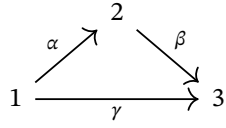
$$
\begin{tikzcd}
                                             & 2 \arrow[rd, "\beta"] &   \\
1 \arrow[rr, "\gamma"'] \arrow[ru, "\alpha"] &                       & 3
\end{tikzcd}
$$
Let $A=kQ$ and $B=kQ/(\beta\alpha)$. Both $A$ and $B$ are gentle one-cycle algebras. Below are the associated marked ribbon graphs $\Gamma_A$ and $\Gamma_B$, respectively, where the cyclic order around each vertex is given clockwise.

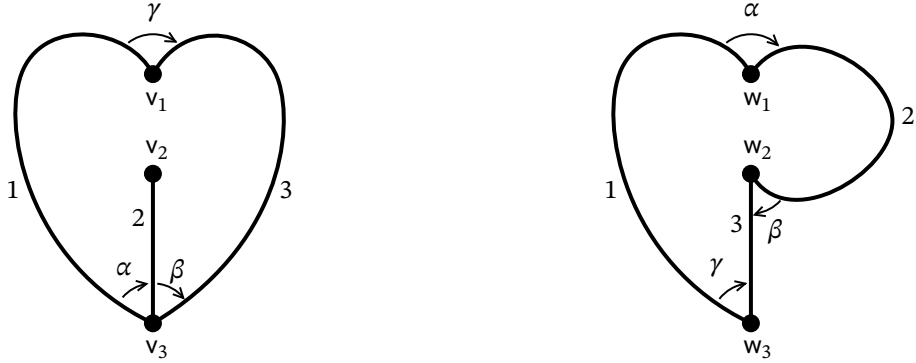
\begin{center}

\tikzset{every picture/.style={line width=0.75pt}} 

\begin{tikzpicture}[x=0.75pt,y=0.75pt,yscale=-1,xscale=1]

\draw  [fill={rgb, 255:red, 0; green, 0; blue, 0 }  ,fill opacity=1 ] (116,85) .. controls (116,82.79) and (117.79,81) .. (120,81) .. controls (122.21,81) and (124,82.79) .. (124,85) .. controls (124,87.21) and (122.21,89) .. (120,89) .. controls (117.79,89) and (116,87.21) .. (116,85) -- cycle ;
\draw  [fill={rgb, 255:red, 0; green, 0; blue, 0 }  ,fill opacity=1 ] (116,135) .. controls (116,132.79) and (117.79,131) .. (120,131) .. controls (122.21,131) and (124,132.79) .. (124,135) .. controls (124,137.21) and (122.21,139) .. (120,139) .. controls (117.79,139) and (116,137.21) .. (116,135) -- cycle ;
\draw  [fill={rgb, 255:red, 0; green, 0; blue, 0 }  ,fill opacity=1 ] (116,210) .. controls (116,207.79) and (117.79,206) .. (120,206) .. controls (122.21,206) and (124,207.79) .. (124,210) .. controls (124,212.21) and (122.21,214) .. (120,214) .. controls (117.79,214) and (116,212.21) .. (116,210) -- cycle ;
\draw [line width=1.5]    (120,135) -- (120,210) ;
\draw [line width=1.5]    (120,85) .. controls (108,61.5) and (67,56.5) .. (55,82.5) .. controls (43,108.5) and (56,179.5) .. (120,210) ;
\draw [line width=1.5]    (120,85) .. controls (137,54.5) and (175,63.5) .. (183,88.5) .. controls (191,113.5) and (184,171.5) .. (120,210) ;
\draw  [fill={rgb, 255:red, 0; green, 0; blue, 0 }  ,fill opacity=1 ] (416,85) .. controls (416,82.79) and (417.79,81) .. (420,81) .. controls (422.21,81) and (424,82.79) .. (424,85) .. controls (424,87.21) and (422.21,89) .. (420,89) .. controls (417.79,89) and (416,87.21) .. (416,85) -- cycle ;
\draw  [fill={rgb, 255:red, 0; green, 0; blue, 0 }  ,fill opacity=1 ] (416,135) .. controls (416,132.79) and (417.79,131) .. (420,131) .. controls (422.21,131) and (424,132.79) .. (424,135) .. controls (424,137.21) and (422.21,139) .. (420,139) .. controls (417.79,139) and (416,137.21) .. (416,135) -- cycle ;
\draw  [fill={rgb, 255:red, 0; green, 0; blue, 0 }  ,fill opacity=1 ] (416,210) .. controls (416,207.79) and (417.79,206) .. (420,206) .. controls (422.21,206) and (424,207.79) .. (424,210) .. controls (424,212.21) and (422.21,214) .. (420,214) .. controls (417.79,214) and (416,212.21) .. (416,210) -- cycle ;
\draw [line width=1.5]    (420,135) -- (420,210) ;
\draw [line width=1.5]    (420,85) .. controls (408,61.5) and (367,56.5) .. (355,82.5) .. controls (343,108.5) and (356,179.5) .. (420,210) ;
\draw [line width=1.5]    (420,85) .. controls (441,53.5) and (491,82.5) .. (491,108.5) .. controls (491,134.5) and (438,166.5) .. (420,135) ;
\draw  [draw opacity=0] (107.63,69.29) .. controls (111.03,66.6) and (115.33,65) .. (120,65) .. controls (124.38,65) and (128.42,66.41) .. (131.72,68.79) -- (120,85) -- cycle ; \draw   (107.63,69.29) .. controls (111.03,66.6) and (115.33,65) .. (120,65) .. controls (124.38,65) and (128.42,66.41) .. (131.72,68.79) ;  
\draw   (128.15,64.1) -- (131.28,68.73) -- (125.7,68.45) ;
\draw  [draw opacity=0] (104.79,197.01) .. controls (107.84,193.45) and (112.11,190.97) .. (116.96,190.23) -- (120,210) -- cycle ; \draw   (104.79,197.01) .. controls (107.84,193.45) and (112.11,190.97) .. (116.96,190.23) ;  
\draw   (111.37,188.72) -- (116.75,190.23) -- (112.32,193.63) ;
\draw  [draw opacity=0] (122.21,190.12) .. controls (127.48,190.7) and (132.14,193.33) .. (135.36,197.19) -- (120,210) -- cycle ; \draw   (122.21,190.12) .. controls (127.48,190.7) and (132.14,193.33) .. (135.36,197.19) ;  
\draw   (133.59,191.96) -- (135.22,197.31) -- (129.97,195.41) ;
\draw  [draw opacity=0] (407.63,69.29) .. controls (411.03,66.6) and (415.33,65) .. (420,65) .. controls (425.3,65) and (430.12,67.06) .. (433.7,70.43) -- (420,85) -- cycle ; \draw   (407.63,69.29) .. controls (411.03,66.6) and (415.33,65) .. (420,65) .. controls (425.3,65) and (430.12,67.06) .. (433.7,70.43) ;  
\draw   (431.37,65.17) -- (433.33,70.4) -- (427.97,68.83) ;
\draw  [draw opacity=0] (434.84,148.41) .. controls (431.69,151.89) and (427.34,154.26) .. (422.43,154.85) -- (420,135) -- cycle ; \draw   (434.84,148.41) .. controls (431.69,151.89) and (427.34,154.26) .. (422.43,154.85) ;  
\draw  [draw opacity=0] (404.19,197.76) .. controls (407.35,193.67) and (412.07,190.85) .. (417.45,190.16) -- (420,210) -- cycle ; \draw   (404.19,197.76) .. controls (407.35,193.67) and (412.07,190.85) .. (417.45,190.16) ;  
\draw   (411.67,188.88) -- (417.13,190.09) -- (412.88,193.73) ;
\draw   (427.49,156.23) -- (421.94,155.53) -- (425.83,151.51) ;

\draw (116,48) node [anchor=north west][inner sep=0.75pt]   [align=left] {$\gamma$};
\draw (100,178) node [anchor=north west][inner sep=0.75pt]   [align=left] {$\alpha$};
\draw (127,176) node [anchor=north west][inner sep=0.75pt]   [align=left] {$\beta$};
\draw (115,93) node [anchor=north west][inner sep=0.75pt]   [align=left] {$\vv_1$};
\draw (115,117) node [anchor=north west][inner sep=0.75pt]   [align=left] {$\vv_2$};
\draw (115,217) node [anchor=north west][inner sep=0.75pt]   [align=left] {$\vv_3$};
\draw (415,48) node [anchor=north west][inner sep=0.75pt]   [align=left] {$\alpha$};
\draw (427,156) node [anchor=north west][inner sep=0.75pt]   [align=left] {$\beta$};
\draw (398,176) node [anchor=north west][inner sep=0.75pt]   [align=left] {$\gamma$};
\draw (415,93) node [anchor=north west][inner sep=0.75pt]   [align=left] {$\ww_1$};
\draw (415,117) node [anchor=north west][inner sep=0.75pt]   [align=left] {$\ww_2$};
\draw (415,217) node [anchor=north west][inner sep=0.75pt]   [align=left] {$\ww_3$};
\draw (45,137) node [anchor=north west][inner sep=0.75pt]   [align=left] {$1$};
\draw (108,151) node [anchor=north west][inner sep=0.75pt]   [align=left] {$2$};
\draw (182,137) node [anchor=north west][inner sep=0.75pt]   [align=left] {$3$};
\draw (345,137) node [anchor=north west][inner sep=0.75pt]   [align=left] {$1$};
\draw (494,100) node [anchor=north west][inner sep=0.75pt]   [align=left] {$2$};
\draw (408,154) node [anchor=north west][inner sep=0.75pt]   [align=left] {$3$};
\end{tikzpicture}
\end{center}
Here $\vv_1=\gamma$, $\vv_2=e_2$, $\vv_3=\beta\alpha$, and $\ww_1=\alpha$, $\ww_2=\beta$, $\ww_3=\gamma$. Moreover, the marking at each vertex corresponds to the incident half-edge which is not the source of an arrow.
\end{eg}

As shown in~\cite{ABCP10,OPS18}, every marked ribbon graph determines an oriented surface with non-empty boundary, called its marked surface. Since we only consider gentle one-cycle algebras, the corresponding surface is in fact an annulus, that is, a closed disk with an open disk removed. We make the following definition.

\begin{dfn}\label{dfn:marked-annulus}
  A \emph{marked annulus} is a pair $(\mathcal{S},M)$, where $\mathcal{S}$ is an annulus whose boundary decomposes as $\partial\mathcal{S}=B_1\sqcup B_2$, and $M$ is a non-empty finite set whose elements are either points on $\partial\mathcal{S}$ or boundary components of $\mathcal{S}$. We require that at most one boundary component belongs to $M$, and that no point of a boundary component $B_i$ is included in $M$ when $B_i\in M$. Elements of $M$ which are points are called \emph{marked points}. If $B_i\in M$, then we call $B_i$ a \emph{puncture}.

This terminology reflects the fact that such a boundary component plays the role of an interior puncture when the annulus is viewed as a disk with one hole collapsed to a point.
\end{dfn}

In fact, the marked ribbon graph of a gentle one-cycle algebra can be embedded into the corresponding marked annulus.

\begin{prop}\cite[Proposition 1.6]{OPS18}\label{prop:embed}
    Let $A$ be a gentle one-cycle algebra, and let $\Gamma_A$ be its associated marked ribbon graph. Then there exists a unique marked annulus $(\mathcal{S}_A,M_A)$ satisfying the following conditions.
\begin{enumerate}
\item The elements in $M_A$ are in bijection with the vertices of $\Gamma_A$.

\item There exists an orientation-preserving embedding
$$
\iota_A:\Gamma_A\longrightarrow \mathcal{S}_A
$$
such that each vertex of $\Gamma_A$ corresponding to a primitive cycle is sent to the puncture in $M_A$, each remaining vertex is sent to the corresponding marked point in $M_A$, and the interior of each edge of $\Gamma_A$ is mapped into the interior of $\mathcal{S}_A$.

\item For each vertex $\vv$ of $\Gamma_A$, the boundary component containing $\iota_A(\vv)$ lies between the half-edges $m(\vv)$ and $\rho(m(\vv))$ in the clockwise orientation.

\item The embedding $\iota_A$ is unique up to homotopy relative to $\partial\mathcal{S}_A$.
\end{enumerate}
\end{prop}

\begin{figure}
\begin{center}
\begin{tikzpicture}[x=0.75pt,y=0.75pt,yscale=-1,xscale=1]

\draw   (20,110) .. controls (20,54.77) and (64.77,10) .. (120,10) .. controls (175.23,10) and (220,54.77) .. (220,110) .. controls (220,165.23) and (175.23,210) .. (120,210) .. controls (64.77,210) and (20,165.23) .. (20,110) -- cycle ;
\draw  [fill={rgb, 255:red, 230; green, 230; blue, 230 }  ,fill opacity=1 ] (95,110) .. controls (95,96.19) and (106.19,85) .. (120,85) .. controls (133.81,85) and (145,96.19) .. (145,110) .. controls (145,123.81) and (133.81,135) .. (120,135) .. controls (106.19,135) and (95,123.81) .. (95,110) -- cycle ;
\draw  [fill={rgb, 255:red, 0; green, 0; blue, 0 }  ,fill opacity=1 ] (116,85) .. controls (116,82.79) and (117.79,81) .. (120,81) .. controls (122.21,81) and (124,82.79) .. (124,85) .. controls (124,87.21) and (122.21,89) .. (120,89) .. controls (117.79,89) and (116,87.21) .. (116,85) -- cycle ;
\draw  [fill={rgb, 255:red, 0; green, 0; blue, 0 }  ,fill opacity=1 ] (116,135) .. controls (116,132.79) and (117.79,131) .. (120,131) .. controls (122.21,131) and (124,132.79) .. (124,135) .. controls (124,137.21) and (122.21,139) .. (120,139) .. controls (117.79,139) and (116,137.21) .. (116,135) -- cycle ;
\draw  [fill={rgb, 255:red, 0; green, 0; blue, 0 }  ,fill opacity=1 ] (116,210) .. controls (116,207.79) and (117.79,206) .. (120,206) .. controls (122.21,206) and (124,207.79) .. (124,210) .. controls (124,212.21) and (122.21,214) .. (120,214) .. controls (117.79,214) and (116,212.21) .. (116,210) -- cycle ;
\draw [line width=1.5]    (120,135) -- (120,210) ;
\draw [line width=1.5]    (120,85) .. controls (108,61.5) and (67,56.5) .. (55,82.5) .. controls (43,108.5) and (56,179.5) .. (120,210) ;
\draw [line width=1.5]    (120,85) .. controls (137,54.5) and (175,63.5) .. (183,88.5) .. controls (191,113.5) and (184,171.5) .. (120,210) ;
\draw   (320,110) .. controls (320,54.77) and (364.77,10) .. (420,10) .. controls (475.23,10) and (520,54.77) .. (520,110) .. controls (520,165.23) and (475.23,210) .. (420,210) .. controls (364.77,210) and (320,165.23) .. (320,110) -- cycle ;
\draw  [fill={rgb, 255:red, 230; green, 230; blue, 230 }  ,fill opacity=1 ] (395,110) .. controls (395,96.19) and (406.19,85) .. (420,85) .. controls (433.81,85) and (445,96.19) .. (445,110) .. controls (445,123.81) and (433.81,135) .. (420,135) .. controls (406.19,135) and (395,123.81) .. (395,110) -- cycle ;
\draw  [fill={rgb, 255:red, 0; green, 0; blue, 0 }  ,fill opacity=1 ] (416,85) .. controls (416,82.79) and (417.79,81) .. (420,81) .. controls (422.21,81) and (424,82.79) .. (424,85) .. controls (424,87.21) and (422.21,89) .. (420,89) .. controls (417.79,89) and (416,87.21) .. (416,85) -- cycle ;
\draw  [fill={rgb, 255:red, 0; green, 0; blue, 0 }  ,fill opacity=1 ] (416,135) .. controls (416,132.79) and (417.79,131) .. (420,131) .. controls (422.21,131) and (424,132.79) .. (424,135) .. controls (424,137.21) and (422.21,139) .. (420,139) .. controls (417.79,139) and (416,137.21) .. (416,135) -- cycle ;
\draw  [fill={rgb, 255:red, 0; green, 0; blue, 0 }  ,fill opacity=1 ] (416,210) .. controls (416,207.79) and (417.79,206) .. (420,206) .. controls (422.21,206) and (424,207.79) .. (424,210) .. controls (424,212.21) and (422.21,214) .. (420,214) .. controls (417.79,214) and (416,212.21) .. (416,210) -- cycle ;
\draw [line width=1.5]    (420,135) -- (420,210) ;
\draw [line width=1.5]    (420,85) .. controls (408,61.5) and (367,56.5) .. (355,82.5) .. controls (343,108.5) and (356,179.5) .. (420,210) ;
\draw [line width=1.5]    (420,85) .. controls (441,53.5) and (491,82.5) .. (491,108.5) .. controls (491,134.5) and (438,166.5) .. (420,135) ;
\draw  [draw opacity=0] (107.63,69.29) .. controls (111.03,66.6) and (115.33,65) .. (120,65) .. controls (124.38,65) and (128.42,66.41) .. (131.72,68.79) -- (120,85) -- cycle ; \draw   (107.63,69.29) .. controls (111.03,66.6) and (115.33,65) .. (120,65) .. controls (124.38,65) and (128.42,66.41) .. (131.72,68.79) ;  
\draw   (128.15,64.1) -- (131.28,68.73) -- (125.7,68.45) ;
\draw  [draw opacity=0] (104.79,197.01) .. controls (107.84,193.45) and (112.11,190.97) .. (116.96,190.23) -- (120,210) -- cycle ; \draw   (104.79,197.01) .. controls (107.84,193.45) and (112.11,190.97) .. (116.96,190.23) ;  
\draw   (111.37,188.72) -- (116.75,190.23) -- (112.32,193.63) ;
\draw  [draw opacity=0] (122.21,190.12) .. controls (127.48,190.7) and (132.14,193.33) .. (135.36,197.19) -- (120,210) -- cycle ; \draw   (122.21,190.12) .. controls (127.48,190.7) and (132.14,193.33) .. (135.36,197.19) ;  
\draw   (133.59,191.96) -- (135.22,197.31) -- (129.97,195.41) ;
\draw  [draw opacity=0] (407.63,69.29) .. controls (411.03,66.6) and (415.33,65) .. (420,65) .. controls (425.3,65) and (430.12,67.06) .. (433.7,70.43) -- (420,85) -- cycle ; \draw   (407.63,69.29) .. controls (411.03,66.6) and (415.33,65) .. (420,65) .. controls (425.3,65) and (430.12,67.06) .. (433.7,70.43) ;  
\draw   (431.37,65.17) -- (433.33,70.4) -- (427.97,68.83) ;
\draw  [draw opacity=0] (434.84,148.41) .. controls (431.69,151.89) and (427.34,154.26) .. (422.43,154.85) -- (420,135) -- cycle ; \draw   (434.84,148.41) .. controls (431.69,151.89) and (427.34,154.26) .. (422.43,154.85) ;  
\draw  [draw opacity=0] (404.19,197.76) .. controls (407.35,193.67) and (412.07,190.85) .. (417.45,190.16) -- (420,210) -- cycle ; \draw   (404.19,197.76) .. controls (407.35,193.67) and (412.07,190.85) .. (417.45,190.16) ;  
\draw   (411.67,188.88) -- (417.13,190.09) -- (412.88,193.73) ;
\draw   (427.49,156.23) -- (421.94,155.53) -- (425.83,151.51) ;

\draw (116,48) node [anchor=north west][inner sep=0.75pt]   [align=left] {$\gamma$};
\draw (100,178) node [anchor=north west][inner sep=0.75pt]   [align=left] {$\alpha$};
\draw (127,176) node [anchor=north west][inner sep=0.75pt]   [align=left] {$\beta$};
\draw (415,48) node [anchor=north west][inner sep=0.75pt]   [align=left] {$\alpha$};
\draw (427,156) node [anchor=north west][inner sep=0.75pt]   [align=left] {$\beta$};
\draw (398,176) node [anchor=north west][inner sep=0.75pt]   [align=left] {$\gamma$};
\draw (3,110) node [anchor=north west][inner sep=0.75pt]   [align=left] {$S_1$};
\draw (80,110) node [anchor=north west][inner sep=0.75pt]   [align=left] {$S_2$};
\draw (145,110) node [anchor=north west][inner sep=0.75pt]   [align=left] {$S_3$};
\draw (300,110) node [anchor=north west][inner sep=0.75pt]   [align=left] {$S_1'$};
\draw (445,110) node [anchor=north west][inner sep=0.75pt]   [align=left] {$S_3'$};
\draw (378,110) node [anchor=north west][inner sep=0.75pt]   [align=left] {$S_2'$};
\end{tikzpicture}
\end{center}  
				\caption{The marked annuli associated with $\Gamma_A$ and $\Gamma_B$ in example \ref{eg:easy-example}, respectively.}
		\label{fig:surface-of-easy-example}	
	\end{figure}
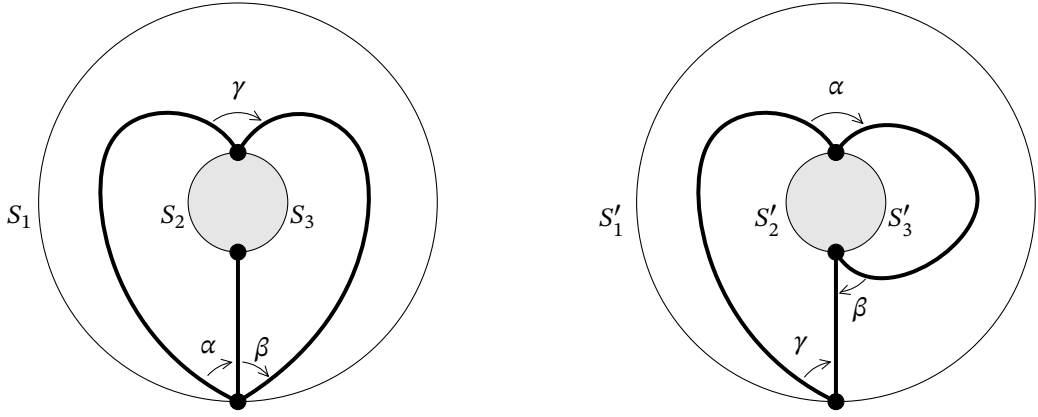	

    \subsection{Gradings and winding numbers}

    To study graded gentle algebras, we need to equip the associated marked surface with a grading defined by a line field. By \cite[Theorem 1.8]{LP20}, for an annulus, or more generally for a genus-zero surface, this grading is determined by the winding numbers of the boundary components. We therefore make the following definition.

\begin{dfn}\label{dfn:graded-marked-annulus}
A \emph{graded marked annulus} is a triple $(\mathcal{S},M,\eta)$, where $(\mathcal{S},M)$ is a marked annulus with $\partial\mathcal{S}=B_1\sqcup B_2$, and
$$
\eta\colon\{B_1,B_2\}\longrightarrow \mathbb{Z}
$$
is a map satisfying $\eta(B_1)=-\eta(B_2)$. Such a map $\eta$ is called a \emph{grading} on $(\mathcal{S},M)$.
\end{dfn}

    In fact, by~\cite{HKK17,OPS18}, every graded gentle one-cycle algebra gives rise to a graded marked annulus. We recall the construction below, following the combinatorial approach in~\cite[Section 3.3]{OZ22} and \cite[Section~3.2]{JSW23}.

    To define the winding number of each boundary component, we first assign winding numbers to the boundary segments between consecutive marked points. Indeed, under the embedding given in Proposition~\ref{prop:embed}, the ribbon graph $\Gamma_A$ decomposes the annulus $\mathcal{S}_A$ into connected components of two types, as shown in Figure~\ref{fig:2-type-components}.

    \begin{figure}
        \begin{center}
\begin{tikzpicture}[x=0.75pt,y=0.75pt,yscale=-1,xscale=1]

\draw  [fill={rgb, 255:red, 0; green, 0; blue, 0 }  ,fill opacity=1 ] (36,116.6) .. controls (36,114.39) and (37.79,112.6) .. (40,112.6) .. controls (42.21,112.6) and (44,114.39) .. (44,116.6) .. controls (44,118.81) and (42.21,120.6) .. (40,120.6) .. controls (37.79,120.6) and (36,118.81) .. (36,116.6) -- cycle ;
\draw  [fill={rgb, 255:red, 0; green, 0; blue, 0 }  ,fill opacity=1 ] (86,30) .. controls (86,27.79) and (87.79,26) .. (90,26) .. controls (92.21,26) and (94,27.79) .. (94,30) .. controls (94,32.21) and (92.21,34) .. (90,34) .. controls (87.79,34) and (86,32.21) .. (86,30) -- cycle ;
\draw  [fill={rgb, 255:red, 0; green, 0; blue, 0 }  ,fill opacity=1 ] (236,116.6) .. controls (236,114.39) and (237.79,112.6) .. (240,112.6) .. controls (242.21,112.6) and (244,114.39) .. (244,116.6) .. controls (244,118.81) and (242.21,120.6) .. (240,120.6) .. controls (237.79,120.6) and (236,118.81) .. (236,116.6) -- cycle ;
\draw  [fill={rgb, 255:red, 0; green, 0; blue, 0 }  ,fill opacity=1 ] (186,30) .. controls (186,27.79) and (187.79,26) .. (190,26) .. controls (192.21,26) and (194,27.79) .. (194,30) .. controls (194,32.21) and (192.21,34) .. (190,34) .. controls (187.79,34) and (186,32.21) .. (186,30) -- cycle ;
\draw [line width=1.5]    (40,116.6) -- (90,203.21) ;
\draw [line width=1.5]    (90,30) -- (40,116.6) ;
\draw [line width=1.5]    (190,203.21) -- (240,116.6) ;
\draw  [draw opacity=0][fill={rgb, 255:red, 230; green, 230; blue, 230 }  ,fill opacity=1 ] (90,203.21) -- (190,203.21) -- (190,213.71) -- (90,213.71) -- cycle ;
\draw  [fill={rgb, 255:red, 0; green, 0; blue, 0 }  ,fill opacity=1 ] (86,203.21) .. controls (86,201) and (87.79,199.21) .. (90,199.21) .. controls (92.21,199.21) and (94,201) .. (94,203.21) .. controls (94,205.41) and (92.21,207.21) .. (90,207.21) .. controls (87.79,207.21) and (86,205.41) .. (86,203.21) -- cycle ;
\draw  [fill={rgb, 255:red, 0; green, 0; blue, 0 }  ,fill opacity=1 ] (186,203.21) .. controls (186,201) and (187.79,199.21) .. (190,199.21) .. controls (192.21,199.21) and (194,201) .. (194,203.21) .. controls (194,205.41) and (192.21,207.21) .. (190,207.21) .. controls (187.79,207.21) and (186,205.41) .. (186,203.21) -- cycle ;
\draw [color={rgb, 255:red, 0; green, 0; blue, 0 }  ,draw opacity=1 ][fill={rgb, 255:red, 3; green, 2; blue, 2 }  ,fill opacity=1 ]   (90,203.21) -- (190,203.21) ;
\draw [line width=1.5]  [dash pattern={on 1.69pt off 2.76pt}]  (90,30) -- (190,30) ;
\draw [line width=1.5]    (190,30) -- (240,116.6) ;
\draw  [draw opacity=0] (52.41,100.92) .. controls (57.03,104.58) and (60,110.25) .. (60,116.6) .. controls (60,122.75) and (57.23,128.25) .. (52.86,131.92) -- (40,116.6) -- cycle ; \draw   (52.41,100.92) .. controls (57.03,104.58) and (60,110.25) .. (60,116.6) .. controls (60,122.75) and (57.23,128.25) .. (52.86,131.92) ;  
\draw   (57.23,130.98) -- (51.9,132.63) -- (53.77,127.37) ;
\draw  [draw opacity=0] (228.12,132.69) .. controls (223.19,129.05) and (220,123.2) .. (220,116.6) .. controls (220,110.09) and (223.12,104.3) .. (227.94,100.65) -- (240,116.6) -- cycle ; \draw   (228.12,132.69) .. controls (223.19,129.05) and (220,123.2) .. (220,116.6) .. controls (220,110.09) and (223.12,104.3) .. (227.94,100.65) ;  
\draw   (222.98,102.4) -- (228.11,100.19) -- (226.79,105.63) ;
\draw  [draw opacity=0] (109.74,33.24) .. controls (108.19,42.75) and (99.94,50) .. (90,50) .. controls (87.71,50) and (85.51,49.62) .. (83.47,48.91) -- (90,30) -- cycle ; \draw   (109.74,33.24) .. controls (108.19,42.75) and (99.94,50) .. (90,50) .. controls (87.71,50) and (85.51,49.62) .. (83.47,48.91) ;  
\draw   (87.85,52.78) -- (83.69,49.04) -- (89.17,47.96) ;
\draw  [draw opacity=0] (197.29,48.63) .. controls (195.03,49.51) and (192.57,50) .. (190,50) .. controls (180.27,50) and (172.17,43.05) .. (170.37,33.85) -- (190,30) -- cycle ; \draw   (197.29,48.63) .. controls (195.03,49.51) and (192.57,50) .. (190,50) .. controls (180.27,50) and (172.17,43.05) .. (170.37,33.85) ;  
\draw   (169.55,39.37) -- (169.94,33.79) -- (174.17,37.45) ;
\draw  [fill={rgb, 255:red, 0; green, 0; blue, 0 }  ,fill opacity=1 ] (336,116.6) .. controls (336,114.39) and (337.79,112.6) .. (340,112.6) .. controls (342.21,112.6) and (344,114.39) .. (344,116.6) .. controls (344,118.81) and (342.21,120.6) .. (340,120.6) .. controls (337.79,120.6) and (336,118.81) .. (336,116.6) -- cycle ;
\draw  [fill={rgb, 255:red, 0; green, 0; blue, 0 }  ,fill opacity=1 ] (386,30) .. controls (386,27.79) and (387.79,26) .. (390,26) .. controls (392.21,26) and (394,27.79) .. (394,30) .. controls (394,32.21) and (392.21,34) .. (390,34) .. controls (387.79,34) and (386,32.21) .. (386,30) -- cycle ;
\draw  [fill={rgb, 255:red, 0; green, 0; blue, 0 }  ,fill opacity=1 ] (536,116.6) .. controls (536,114.39) and (537.79,112.6) .. (540,112.6) .. controls (542.21,112.6) and (544,114.39) .. (544,116.6) .. controls (544,118.81) and (542.21,120.6) .. (540,120.6) .. controls (537.79,120.6) and (536,118.81) .. (536,116.6) -- cycle ;
\draw  [fill={rgb, 255:red, 0; green, 0; blue, 0 }  ,fill opacity=1 ] (486,30) .. controls (486,27.79) and (487.79,26) .. (490,26) .. controls (492.21,26) and (494,27.79) .. (494,30) .. controls (494,32.21) and (492.21,34) .. (490,34) .. controls (487.79,34) and (486,32.21) .. (486,30) -- cycle ;
\draw [line width=1.5]    (340,116.6) -- (390,203.21) ;
\draw [line width=1.5]    (390,30) -- (340,116.6) ;
\draw [line width=1.5]    (490,203.21) -- (540,116.6) ;
\draw  [fill={rgb, 255:red, 0; green, 0; blue, 0 }  ,fill opacity=1 ] (386,203.21) .. controls (386,201) and (387.79,199.21) .. (390,199.21) .. controls (392.21,199.21) and (394,201) .. (394,203.21) .. controls (394,205.41) and (392.21,207.21) .. (390,207.21) .. controls (387.79,207.21) and (386,205.41) .. (386,203.21) -- cycle ;
\draw  [fill={rgb, 255:red, 0; green, 0; blue, 0 }  ,fill opacity=1 ] (486,203.21) .. controls (486,201) and (487.79,199.21) .. (490,199.21) .. controls (492.21,199.21) and (494,201) .. (494,203.21) .. controls (494,205.41) and (492.21,207.21) .. (490,207.21) .. controls (487.79,207.21) and (486,205.41) .. (486,203.21) -- cycle ;
\draw [color={rgb, 255:red, 0; green, 0; blue, 0 }  ,draw opacity=1 ][fill={rgb, 255:red, 3; green, 2; blue, 2 }  ,fill opacity=1 ][line width=1.5]    (390,203.21) -- (490,203.21) ;
\draw [line width=1.5]  [dash pattern={on 1.69pt off 2.76pt}]  (390,30) -- (490,30) ;
\draw [line width=1.5]    (490,30) -- (540,116.6) ;
\draw  [draw opacity=0] (352.41,100.92) .. controls (357.03,104.58) and (360,110.25) .. (360,116.6) .. controls (360,122.75) and (357.23,128.25) .. (352.86,131.92) -- (340,116.6) -- cycle ; \draw   (352.41,100.92) .. controls (357.03,104.58) and (360,110.25) .. (360,116.6) .. controls (360,122.75) and (357.23,128.25) .. (352.86,131.92) ;  
\draw   (357.23,130.98) -- (351.9,132.63) -- (353.77,127.37) ;
\draw  [draw opacity=0] (528.12,132.69) .. controls (523.19,129.05) and (520,123.2) .. (520,116.6) .. controls (520,110.09) and (523.12,104.3) .. (527.94,100.65) -- (540,116.6) -- cycle ; \draw   (528.12,132.69) .. controls (523.19,129.05) and (520,123.2) .. (520,116.6) .. controls (520,110.09) and (523.12,104.3) .. (527.94,100.65) ;  
\draw   (522.98,102.4) -- (528.11,100.19) -- (526.79,105.63) ;
\draw  [draw opacity=0] (409.74,33.24) .. controls (408.19,42.75) and (399.94,50) .. (390,50) .. controls (387.71,50) and (385.51,49.62) .. (383.47,48.91) -- (390,30) -- cycle ; \draw   (409.74,33.24) .. controls (408.19,42.75) and (399.94,50) .. (390,50) .. controls (387.71,50) and (385.51,49.62) .. (383.47,48.91) ;  
\draw   (387.85,52.78) -- (383.69,49.04) -- (389.17,47.96) ;
\draw  [draw opacity=0] (497.29,48.63) .. controls (495.03,49.51) and (492.57,50) .. (490,50) .. controls (480.27,50) and (472.17,43.05) .. (470.37,33.85) -- (490,30) -- cycle ; \draw   (497.29,48.63) .. controls (495.03,49.51) and (492.57,50) .. (490,50) .. controls (480.27,50) and (472.17,43.05) .. (470.37,33.85) ;  
\draw   (469.55,39.37) -- (469.94,33.79) -- (474.17,37.45) ;
\draw  [fill={rgb, 255:red, 230; green, 230; blue, 230 }  ,fill opacity=1 ] (415,116.6) .. controls (415,102.8) and (426.19,91.6) .. (440,91.6) .. controls (453.81,91.6) and (465,102.8) .. (465,116.6) .. controls (465,130.41) and (453.81,141.6) .. (440,141.6) .. controls (426.19,141.6) and (415,130.41) .. (415,116.6) -- cycle ;
\draw  [draw opacity=0] (383.46,184.3) .. controls (385.51,183.59) and (387.71,183.21) .. (390,183.21) .. controls (400.03,183.21) and (408.33,190.59) .. (409.78,200.21) -- (390,203.21) -- cycle ; \draw   (383.46,184.3) .. controls (385.51,183.59) and (387.71,183.21) .. (390,183.21) .. controls (400.03,183.21) and (408.33,190.59) .. (409.78,200.21) ;  
\draw  [draw opacity=0] (470.18,200.54) .. controls (471.48,190.75) and (479.86,183.21) .. (490,183.21) .. controls (493.01,183.21) and (495.86,183.87) .. (498.42,185.06) -- (490,203.21) -- cycle ; \draw   (470.18,200.54) .. controls (471.48,190.75) and (479.86,183.21) .. (490,183.21) .. controls (493.01,183.21) and (495.86,183.87) .. (498.42,185.06) ;  
\draw   (411.25,194.59) -- (409.98,200.03) -- (406.39,195.75) ;
\draw   (494.32,180.94) -- (498.27,184.89) -- (492.74,185.69) ;

\draw (60,109) node [anchor=north west][inner sep=0.75pt]   [align=left] {$\alpha_1$};
\draw (200,109) node [anchor=north west][inner sep=0.75pt]   [align=left] {$\alpha_n$};
\draw (134,185) node [anchor=north west][inner sep=0.75pt]   [align=left] {$S$};
\draw (360,109) node [anchor=north west][inner sep=0.75pt]   [align=left] {$\alpha_2$};
\draw (492,109) node [anchor=north west][inner sep=0.75pt]   [align=left] {$\alpha_{n-1}$};
\draw (393,170) node [anchor=north west][inner sep=0.75pt]   [align=left] {$\alpha_1$};
\draw (467,170) node [anchor=north west][inner sep=0.75pt]   [align=left] {$\alpha_n$};
\draw (400,110) node [anchor=north west][inner sep=0.75pt]   [align=left] {$S$};
\end{tikzpicture}
        \end{center}
        \caption{Two types of connected components of $\mathcal{S}_A\setminus \Gamma_A$.}
        \label{fig:2-type-components}
    \end{figure}
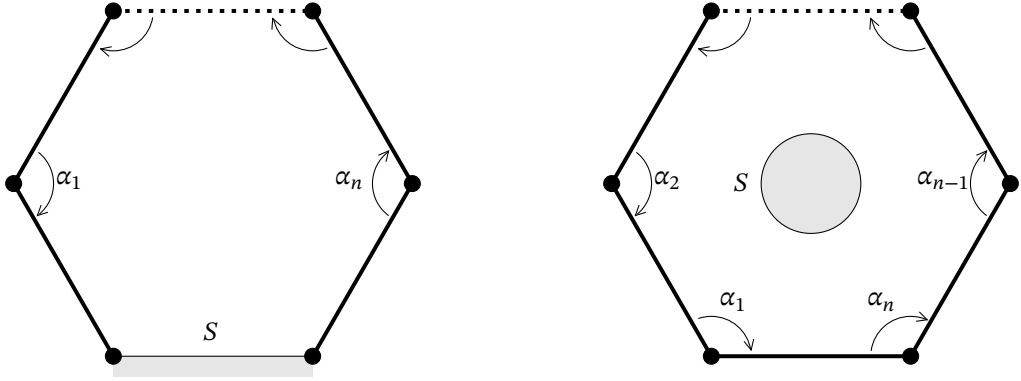

    Note that the connected components described above can be used to determine whether a gentle algebra has finite global dimension.

    \begin{prop}\cite[Theorem 1.1]{LGH24}\label{prop:gldim}
Let $A$ be a gentle one-cycle algebra. Then $A$ has finite global dimension if and only if there is no connected component of the second type in Figure~\ref{fig:2-type-components}.
    \end{prop}

    We now give the following definition.

    \begin{dfn}\label{dfn:winding-number}
Let $A$ be a graded gentle one-cycle algebra, and let
$(\mathcal{S}_A,M_A)$ be its associated marked annulus, with
$\partial\mathcal{S}_A=B_1\sqcup B_2$. We define the
\emph{winding numbers} of the boundary components as follows.

\begin{enumerate}
    \item Let $S$ be a boundary segment between two consecutive marked points.
    Then $S$ lies on the boundary of a polygon of the first type in
    Figure~\ref{fig:2-type-components}. We define the winding number of $S$ by
    $$
    \omega_\eta(S)=|\alpha_1|+\cdots+|\alpha_n|+1-n.
    $$

   \item Suppose that $B_i$ is a boundary component such that $M_A\cap B_i\neq\emptyset$ and $B_i\notin M_A$. Then the marked points cut $B_i$ into boundary segments $S_1,\ldots,S_m$. Let $\beta_1,\ldots,\beta_\ell$ be the arrows around the marked points in $M_A\cap B_i$. We define the winding number of $B_i$ by
$$\omega_\eta(B_i)=\omega_\eta(S_1)+\cdots+\omega_\eta(S_m)-|\beta_1|-\cdots-|\beta_\ell|.$$

\item Suppose that $B_i$ is a boundary component such that $M_A\cap B_i=\emptyset$. Then $B_i$ is a connected component of the second type in Figure~\ref{fig:2-type-components}. We define the winding number of $B_i$ by
$$\omega_\eta(B_i)=|\alpha_1|+\cdots+|\alpha_n|-n.$$

\item Suppose that $B_i\in M_A$, that is, $B_i$ is a puncture. Let $\beta_1,\ldots,\beta_\ell$ be the arrows around this puncture. We define the winding number of $B_i$ by
$$\omega_\eta(B_i)=-|\beta_1|-\cdots-|\beta_\ell|.$$
\end{enumerate}
    \end{dfn}

    The above definition associates a graded marked annulus to each graded gentle one-cycle algebra.

    \begin{prop}
        Let $A$ be a graded gentle one-cycle algebra. Then $(\mathcal{S}_A,M_A,\eta_A:=\omega_\eta)$ with
$\partial\mathcal{S}_A=B_1\sqcup B_2$ is a graded marked annulus.
    \end{prop}

    \begin{proof}
        This is natural by the Poincar\'{e}-Hopf theorem (see for example,~\cite[Theorem 3.18]{OZ22}). Indeed, for an annulus we have genus $g=0$ and the number of boundary components is $b=2$. Hence $$\omega_\eta(B_1)+\omega_\eta(B_2)=2(2-2g-b)=0.$$
    \end{proof}

    The following example illustrates the construction of the winding numbers.

    \begin{eg} (Example \ref{eg:easy-example} revisited)
    For the boundary segments shown in Figure~\ref{fig:surface-of-easy-example}, we have:
$$\begin{aligned}
\omega_\eta(S_1) &= |\gamma|, &
\omega_\eta(S_2) &= |\alpha|, &
\omega_\eta(S_3) &= |\beta|,\\
\omega_\eta(S'_1) &= |\alpha|+|\beta|-1, &
\omega_\eta(S'_2) &= |\gamma|, &
\omega_\eta(S'_3) &= 1.
\end{aligned}$$
    Let $B_1,B_2$ be the boundary components of $\mathcal{S}_A$ corresponding to $S_1$ and to $S_2,S_3$, respectively. Similarly, let $B'_1,B'_2$ be the boundary components of $\mathcal{S}_B$ corresponding to $S'_1$ and to $S'_2,S'_3$, respectively. Then we have:
$$\begin{aligned}
\omega_\eta(B_1)  &= -\omega_\eta(B_2)  = |\gamma|-|\alpha|-|\beta|,\\
\omega_\eta(B'_1) &= -\omega_\eta(B'_2) = |\alpha|+|\beta|-1-|\gamma|.
\end{aligned}$$
    \end{eg}

    \subsection{Derived equivalences of graded gentle one-cycle algebras}

    We are now ready to show that the derived equivalences between graded gentle one-cycle algebras can be completely characterized in terms of their associated graded marked annuli.

    \begin{thm}\label{thm:classify-graded-gentle}
       Let $A_1$ and $A_2$ be graded gentle one-cycle algebras, and let $(\mathcal{S}_{A_1},M_{A_1},\eta_{A_1})$ with $\partial\mathcal{S}_{A_1}=B_1\sqcup B_2$ and $(\mathcal{S}_{A_2},M_{A_2},\eta_{A_2})$ with $\partial\mathcal{S}_{A_2}=B'_1\sqcup B'_2$ be their associated graded marked annuli. Then $A_1$ and $A_2$ are derived equivalent if and only if, up to a permutation of the two boundary components, one has
        $$|B_i\cap M_{A_1}|=|B'_i\cap M_{A_2}|\quad\text{and}\quad\eta_{A_1}(B_i)=\eta_{A_2}(B'_i)$$ for $i=1,2$.
    \end{thm}

    \begin{proof}
      Indeed, for every graded gentle one-cycle algebra, at least one of the following holds: it is finite-dimensional, or it has finite global dimension. Thus this theorem follows directly from \cite[Theorem~C]{Opp25}. Alternatively, the finite global dimension case follows from \cite[Theorem~1.4]{JSW23}, whereas the finite-dimensional case follows from \cite[Theorem~1.2]{KY18}.
    \end{proof}

    \begin{cor}\label{cor:standard-form}
         Let $A$ be a non-positively graded gentle one-cycle algebra.
\begin{enumerate}
\item If $A$ has finite global dimension, then $A$ is derived equivalent to the graded path algebra $kQ$, where $Q$ is the quiver
$$
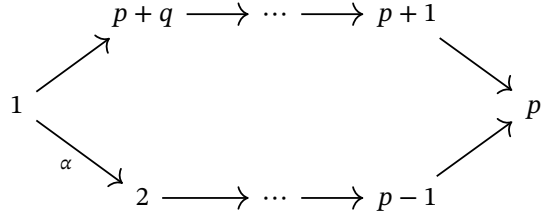
\begin{tikzcd} & p+q \arrow[r] & \cdots \arrow[r] & p+1 \arrow[rd] & \\ 1 \arrow[rd, "\alpha"'] \arrow[ru] & & & & p \\ & 2 \arrow[r] & \cdots \arrow[r] & p-1 \arrow[ru] & \end{tikzcd}$$
where $p\in\mathbb{Z}_{>0}$, $q\in\mathbb{Z}_{\geq -1}$, and $|\alpha|\in\mathbb{Z}_{\leq 0}$, while all other arrows have degree zero. Note that when $q=-1$, the quiver is a $p$-cycle.

\item If $A$ has infinite global dimension, then $A$ is derived equivalent to the graded quiver algebra $kQ/I$, where $Q$ is the quiver
$$
\begin{tikzcd}
1 \arrow["\alpha"', loop, distance=2em, in=215, out=145] \arrow[r] & 2 \arrow[r] & \cdots \arrow[r] & p
\end{tikzcd}
$$
with $I=\langle\alpha^2\rangle$, $|\alpha|\in\mathbb{Z}_{\leq 0}$, while all other arrows have degree zero.
\end{enumerate}
    \end{cor}

    \begin{proof}
   By Theorem~\ref{thm:classify-graded-gentle}, it suffices to choose, for each graded marked annulus, a representative marked ribbon graph. We choose representatives as follows:
\begin{enumerate}
\item For infinite-dimensional gentle one-cycle algebras, since the grading is non-positive, the winding number of the puncture is a non-negative integer, and we take the surface model shown in Figure~\ref{fig:representative-algebras-infinite}.
\item For finite-dimensional gentle one-cycle algebras of finite global dimension, we take the surface model shown on the left in Figure~\ref{fig:representative-algebras}, where the winding number of the middle boundary component can be chosen as a non-positive integer.
\item For finite-dimensional gentle one-cycle algebras of infinite global dimension, we take the surface model shown on the right in Figure~\ref{fig:representative-algebras}. Since the grading is non-positive, the winding number of the boundary component without marked points is necessarily a negative integer.
\end{enumerate}
     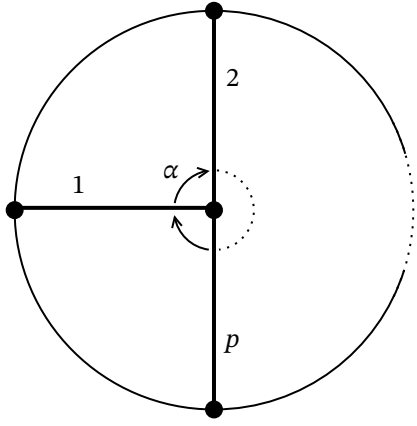
\begin{figure}
        \begin{center}

\tikzset{every picture/.style={line width=0.75pt}} 

\begin{tikzpicture}[x=0.75pt,y=0.75pt,yscale=-1,xscale=1]

\draw  [fill={rgb, 255:red, 0; green, 0; blue, 0 }  ,fill opacity=1 ] (116,110) .. controls (116,107.79) and (117.79,106) .. (120,106) .. controls (122.21,106) and (124,107.79) .. (124,110) .. controls (124,112.21) and (122.21,114) .. (120,114) .. controls (117.79,114) and (116,112.21) .. (116,110) -- cycle ;
\draw  [draw opacity=0] (215.39,140.11) .. controls (202.61,180.63) and (164.74,210) .. (120,210) .. controls (64.77,210) and (20,165.23) .. (20,110) .. controls (20,54.77) and (64.77,10) .. (120,10) .. controls (165.42,10) and (203.77,40.28) .. (215.96,81.76) -- (120,110) -- cycle ; \draw   (215.39,140.11) .. controls (202.61,180.63) and (164.74,210) .. (120,210) .. controls (64.77,210) and (20,165.23) .. (20,110) .. controls (20,54.77) and (64.77,10) .. (120,10) .. controls (165.42,10) and (203.77,40.28) .. (215.96,81.76) ;  
\draw  [fill={rgb, 255:red, 0; green, 0; blue, 0 }  ,fill opacity=1 ] (116,210) .. controls (116,207.79) and (117.79,206) .. (120,206) .. controls (122.21,206) and (124,207.79) .. (124,210) .. controls (124,212.21) and (122.21,214) .. (120,214) .. controls (117.79,214) and (116,212.21) .. (116,210) -- cycle ;
\draw  [fill={rgb, 255:red, 0; green, 0; blue, 0 }  ,fill opacity=1 ] (16,110) .. controls (16,107.79) and (17.79,106) .. (20,106) .. controls (22.21,106) and (24,107.79) .. (24,110) .. controls (24,112.21) and (22.21,114) .. (20,114) .. controls (17.79,114) and (16,112.21) .. (16,110) -- cycle ;
\draw  [fill={rgb, 255:red, 0; green, 0; blue, 0 }  ,fill opacity=1 ] (116,10) .. controls (116,7.79) and (117.79,6) .. (120,6) .. controls (122.21,6) and (124,7.79) .. (124,10) .. controls (124,12.21) and (122.21,14) .. (120,14) .. controls (117.79,14) and (116,12.21) .. (116,10) -- cycle ;
\draw  [draw opacity=0][dash pattern={on 0.84pt off 2.51pt}] (209.84,66.03) .. controls (216.35,79.3) and (220,94.22) .. (220,110) .. controls (220,125.98) and (216.25,141.09) .. (209.58,154.49) -- (120,110) -- cycle ; \draw  [dash pattern={on 0.84pt off 2.51pt}] (209.84,66.03) .. controls (216.35,79.3) and (220,94.22) .. (220,110) .. controls (220,125.98) and (216.25,141.09) .. (209.58,154.49) ;  
\draw [line width=1.5]    (18.86,108.86) -- (118.86,108.86) ;
\draw [line width=1.5]    (120,10) -- (120,110) ;
\draw [line width=1.5]    (120,110) -- (120,210) ;
\draw  [draw opacity=0] (100.31,106.46) .. controls (101.76,98.33) and (108.13,91.9) .. (116.23,90.35) -- (120,110) -- cycle ; \draw   (100.31,106.46) .. controls (101.76,98.33) and (108.13,91.9) .. (116.23,90.35) ;  
\draw  [draw opacity=0] (117.4,129.83) .. controls (108.7,128.7) and (101.75,121.98) .. (100.29,113.39) -- (120,110) -- cycle ; \draw   (117.4,129.83) .. controls (108.7,128.7) and (101.75,121.98) .. (100.29,113.39) ;  
\draw  [draw opacity=0][dash pattern={on 0.84pt off 2.51pt}] (121.45,90.05) .. controls (131.82,90.8) and (140,99.44) .. (140,110) .. controls (140,120.2) and (132.37,128.61) .. (122.51,129.84) -- (120,110) -- cycle ; \draw  [dash pattern={on 0.84pt off 2.51pt}] (121.45,90.05) .. controls (131.82,90.8) and (140,99.44) .. (140,110) .. controls (140,120.2) and (132.37,128.61) .. (122.51,129.84) ;  
\draw   (111.46,88.79) -- (116.78,90.5) -- (112.22,93.73) ;
\draw   (99,119.02) -- (100.32,113.59) -- (103.88,117.9) ;

\draw (47.33,91.33) node [anchor=north west][inner sep=0.75pt]   [align=left] {$1$};
\draw (124.67,37.33) node [anchor=north west][inner sep=0.75pt]   [align=left] {$2$};
\draw (123.33,171) node [anchor=north west][inner sep=0.75pt]   [align=left] {$p$};
\draw (93,86) node [anchor=north west][inner sep=0.75pt]   [align=left] {$\alpha$};
\end{tikzpicture}
        \end{center}
            \caption{Surface models for infinite-dimensional gentle one-cycle algebras.}
            \label{fig:representative-algebras-infinite}
        \end{figure}
        \begin{figure}
        \begin{center}
\begin{tikzpicture}[x=0.75pt,y=0.75pt,yscale=-1,xscale=1]

\draw  [draw opacity=0] (218.45,146.93) .. controls (204.84,185.98) and (167.69,214) .. (124,214) .. controls (68.77,214) and (24,169.23) .. (24,114) .. controls (24,58.77) and (68.77,14) .. (124,14) .. controls (172.37,14) and (212.72,48.34) .. (221.99,93.97) -- (124,114) -- cycle ; \draw   (218.45,146.93) .. controls (204.84,185.98) and (167.69,214) .. (124,214) .. controls (68.77,214) and (24,169.23) .. (24,114) .. controls (24,58.77) and (68.77,14) .. (124,14) .. controls (172.37,14) and (212.72,48.34) .. (221.99,93.97) ;  
\draw  [draw opacity=0][fill={rgb, 255:red, 230; green, 230; blue, 230 }  ,fill opacity=1 ] (148.15,120.49) .. controls (145.29,131.15) and (135.56,139) .. (124,139) .. controls (110.19,139) and (99,127.81) .. (99,114) .. controls (99,100.19) and (110.19,89) .. (124,89) .. controls (133.97,89) and (142.58,94.84) .. (146.59,103.28) -- (124,114) -- cycle ; \draw   (148.15,120.49) .. controls (145.29,131.15) and (135.56,139) .. (124,139) .. controls (110.19,139) and (99,127.81) .. (99,114) .. controls (99,100.19) and (110.19,89) .. (124,89) .. controls (133.97,89) and (142.58,94.84) .. (146.59,103.28) ;  
\draw  [draw opacity=0][fill={rgb, 255:red, 230; green, 230; blue, 230 }  ,fill opacity=1 ][dash pattern={on 0.84pt off 2.51pt}] (146.33,102.75) .. controls (148.04,106.13) and (149,109.95) .. (149,114) .. controls (149,116.34) and (148.68,118.61) .. (148.07,120.76) -- (124,114) -- cycle ; \draw  [dash pattern={on 0.84pt off 2.51pt}] (146.33,102.75) .. controls (148.04,106.13) and (149,109.95) .. (149,114) .. controls (149,116.34) and (148.68,118.61) .. (148.07,120.76) ;  
\draw  [draw opacity=0][dash pattern={on 0.84pt off 2.51pt}] (214.18,70.73) .. controls (220.47,83.83) and (224,98.5) .. (224,114) .. controls (224,132.18) and (219.15,149.23) .. (210.67,163.92) -- (124,114) -- cycle ; \draw  [dash pattern={on 0.84pt off 2.51pt}] (214.18,70.73) .. controls (220.47,83.83) and (224,98.5) .. (224,114) .. controls (224,132.18) and (219.15,149.23) .. (210.67,163.92) ;  
\draw  [fill={rgb, 255:red, 0; green, 0; blue, 0 }  ,fill opacity=1 ] (120,89) .. controls (120,86.79) and (121.79,85) .. (124,85) .. controls (126.21,85) and (128,86.79) .. (128,89) .. controls (128,91.21) and (126.21,93) .. (124,93) .. controls (121.79,93) and (120,91.21) .. (120,89) -- cycle ;
\draw  [fill={rgb, 255:red, 0; green, 0; blue, 0 }  ,fill opacity=1 ] (95,114) .. controls (95,111.79) and (96.79,110) .. (99,110) .. controls (101.21,110) and (103,111.79) .. (103,114) .. controls (103,116.21) and (101.21,118) .. (99,118) .. controls (96.79,118) and (95,116.21) .. (95,114) -- cycle ;
\draw  [fill={rgb, 255:red, 0; green, 0; blue, 0 }  ,fill opacity=1 ] (120,139) .. controls (120,136.79) and (121.79,135) .. (124,135) .. controls (126.21,135) and (128,136.79) .. (128,139) .. controls (128,141.21) and (126.21,143) .. (124,143) .. controls (121.79,143) and (120,141.21) .. (120,139) -- cycle ;
\draw  [fill={rgb, 255:red, 0; green, 0; blue, 0 }  ,fill opacity=1 ] (120,214) .. controls (120,211.79) and (121.79,210) .. (124,210) .. controls (126.21,210) and (128,211.79) .. (128,214) .. controls (128,216.21) and (126.21,218) .. (124,218) .. controls (121.79,218) and (120,216.21) .. (120,214) -- cycle ;
\draw  [fill={rgb, 255:red, 0; green, 0; blue, 0 }  ,fill opacity=1 ] (20,114) .. controls (20,111.79) and (21.79,110) .. (24,110) .. controls (26.21,110) and (28,111.79) .. (28,114) .. controls (28,116.21) and (26.21,118) .. (24,118) .. controls (21.79,118) and (20,116.21) .. (20,114) -- cycle ;
\draw  [fill={rgb, 255:red, 0; green, 0; blue, 0 }  ,fill opacity=1 ] (120,14) .. controls (120,11.79) and (121.79,10) .. (124,10) .. controls (126.21,10) and (128,11.79) .. (128,14) .. controls (128,16.21) and (126.21,18) .. (124,18) .. controls (121.79,18) and (120,16.21) .. (120,14) -- cycle ;
\draw [line width=1.5]    (124,89) .. controls (140.18,76.86) and (188.4,72.8) .. (195.6,107.6) .. controls (202.8,142.4) and (172.59,187.56) .. (124,214) ;
\draw [line width=1.5]    (124,89) .. controls (114.44,83.61) and (60.8,62.4) .. (50.8,106.8) .. controls (40.8,151.2) and (64.4,172.8) .. (80,186.8) .. controls (95.6,200.8) and (114.09,207.42) .. (124,214) ;
\draw [line width=1.5]    (99,114) .. controls (82,108.8) and (105.6,185.6) .. (124,214) ;
\draw [line width=1.5]    (124,139) -- (124,214) ;
\draw [line width=1.5]    (24,114) .. controls (39.68,85.64) and (51.2,61.2) .. (77.2,57.6) .. controls (103.2,54) and (116.54,76.72) .. (124,89) ;
\draw [line width=1.5]    (124,14) -- (124,89) ;
\draw  [draw opacity=0] (106.25,79.77) .. controls (107.44,77.49) and (109.05,75.47) .. (110.98,73.81) -- (124,89) -- cycle ; \draw   (106.25,79.77) .. controls (107.44,77.49) and (109.05,75.47) .. (110.98,73.81) ;  
\draw  [draw opacity=0] (114.85,71.21) .. controls (117.05,70.08) and (119.49,69.34) .. (122.07,69.09) -- (124,89) -- cycle ; \draw   (114.85,71.21) .. controls (117.05,70.08) and (119.49,69.34) .. (122.07,69.09) ;  
\draw  [draw opacity=0][dash pattern={on 0.84pt off 2.51pt}] (126.44,69.15) .. controls (133.08,69.96) and (138.72,74.02) .. (141.7,79.69) -- (124,89) -- cycle ; \draw  [dash pattern={on 0.84pt off 2.51pt}] (126.44,69.15) .. controls (133.08,69.96) and (138.72,74.02) .. (141.7,79.69) ;  
\draw  [draw opacity=0] (107.72,202.37) .. controls (109.13,200.41) and (110.89,198.7) .. (112.9,197.36) -- (124,214) -- cycle ; \draw   (107.72,202.37) .. controls (109.13,200.41) and (110.89,198.7) .. (112.9,197.36) ;  
\draw  [draw opacity=0] (116.36,195.51) .. controls (118.15,194.77) and (120.06,194.28) .. (122.07,194.09) -- (124,214) -- cycle ; \draw   (116.36,195.51) .. controls (118.15,194.77) and (120.06,194.28) .. (122.07,194.09) ;  
\draw  [draw opacity=0][dash pattern={on 0.84pt off 2.51pt}] (126.44,194.15) .. controls (131.98,194.82) and (136.81,197.76) .. (140,202) -- (124,214) -- cycle ; \draw  [dash pattern={on 0.84pt off 2.51pt}] (126.44,194.15) .. controls (131.98,194.82) and (136.81,197.76) .. (140,202) ;  
\draw   (105.76,75.29) -- (111.07,73.54) -- (109.29,78.84) ;
\draw   (116.78,67.7) -- (122.16,69.24) -- (117.7,72.62) ;
\draw   (117.09,192.65) -- (122.53,193.92) -- (118.25,197.51) ;
\draw   (107.76,198.89) -- (113.07,197.14) -- (111.29,202.44) ;
\draw  [draw opacity=0] (518.45,146.93) .. controls (504.84,185.98) and (467.69,214) .. (424,214) .. controls (368.77,214) and (324,169.23) .. (324,114) .. controls (324,58.77) and (368.77,14) .. (424,14) .. controls (472.37,14) and (512.72,48.34) .. (521.99,93.97) -- (424,114) -- cycle ; \draw   (518.45,146.93) .. controls (504.84,185.98) and (467.69,214) .. (424,214) .. controls (368.77,214) and (324,169.23) .. (324,114) .. controls (324,58.77) and (368.77,14) .. (424,14) .. controls (472.37,14) and (512.72,48.34) .. (521.99,93.97) ;  
\draw  [draw opacity=0][dash pattern={on 0.84pt off 2.51pt}] (514.18,70.73) .. controls (520.47,83.83) and (524,98.5) .. (524,114) .. controls (524,132.18) and (519.15,149.23) .. (510.67,163.92) -- (424,114) -- cycle ; \draw  [dash pattern={on 0.84pt off 2.51pt}] (514.18,70.73) .. controls (520.47,83.83) and (524,98.5) .. (524,114) .. controls (524,132.18) and (519.15,149.23) .. (510.67,163.92) ;  
\draw  [fill={rgb, 255:red, 0; green, 0; blue, 0 }  ,fill opacity=1 ] (420,214) .. controls (420,211.79) and (421.79,210) .. (424,210) .. controls (426.21,210) and (428,211.79) .. (428,214) .. controls (428,216.21) and (426.21,218) .. (424,218) .. controls (421.79,218) and (420,216.21) .. (420,214) -- cycle ;
\draw  [fill={rgb, 255:red, 0; green, 0; blue, 0 }  ,fill opacity=1 ] (320,114) .. controls (320,111.79) and (321.79,110) .. (324,110) .. controls (326.21,110) and (328,111.79) .. (328,114) .. controls (328,116.21) and (326.21,118) .. (324,118) .. controls (321.79,118) and (320,116.21) .. (320,114) -- cycle ;
\draw  [fill={rgb, 255:red, 0; green, 0; blue, 0 }  ,fill opacity=1 ] (420,14) .. controls (420,11.79) and (421.79,10) .. (424,10) .. controls (426.21,10) and (428,11.79) .. (428,14) .. controls (428,16.21) and (426.21,18) .. (424,18) .. controls (421.79,18) and (420,16.21) .. (420,14) -- cycle ;
\draw  [fill={rgb, 255:red, 230; green, 230; blue, 230 }  ,fill opacity=1 ] (399,114) .. controls (399,100.19) and (410.19,89) .. (424,89) .. controls (437.81,89) and (449,100.19) .. (449,114) .. controls (449,127.81) and (437.81,139) .. (424,139) .. controls (410.19,139) and (399,127.81) .. (399,114) -- cycle ;
\draw [line width=1.5]    (324,114) .. controls (324.57,110.97) and (401.67,80) .. (435.67,81.33) .. controls (469.67,82.67) and (469.67,152.67) .. (428.67,146) .. controls (387.67,139.33) and (346,119.33) .. (328,114) ;
\draw [line width=1.5]    (324,114) .. controls (328.42,129.06) and (338.86,157.8) .. (359.67,178) .. controls (380.47,198.2) and (410.4,208.93) .. (424,214) ;
\draw [line width=1.5]    (324,114) .. controls (337.17,124.67) and (341.83,133.67) .. (376.17,151) .. controls (410.5,168.33) and (477,184) .. (493.67,151.33) .. controls (510.33,118.67) and (444.09,27.33) .. (424,14) ;
\draw  [draw opacity=0] (342.83,107.25) .. controls (343.59,109.36) and (344,111.63) .. (344,114) .. controls (344,115.15) and (343.9,116.27) .. (343.72,117.37) -- (324,114) -- cycle ; \draw   (342.83,107.25) .. controls (343.59,109.36) and (344,111.63) .. (344,114) .. controls (344,115.15) and (343.9,116.27) .. (343.72,117.37) ;  
\draw   (346.53,112.53) -- (343.9,117.46) -- (341.53,112.4) ;
\draw  [draw opacity=0] (342.61,121.33) .. controls (341.95,123.02) and (341.06,124.6) .. (339.98,126.03) -- (324,114) -- cycle ; \draw   (342.61,121.33) .. controls (341.95,123.02) and (341.06,124.6) .. (339.98,126.03) ;  
\draw   (344.66,122.54) -- (340.16,125.86) -- (340.21,120.27) ;
\draw  [draw opacity=0][dash pattern={on 0.84pt off 2.51pt}] (337.13,129.09) .. controls (335.35,130.63) and (333.3,131.87) .. (331.06,132.72) -- (324,114) -- cycle ; \draw  [dash pattern={on 0.84pt off 2.51pt}] (337.13,129.09) .. controls (335.35,130.63) and (333.3,131.87) .. (331.06,132.72) ;  

\draw (96,184) node [anchor=north west][inner sep=0.75pt]   [align=left] {$\alpha$};
\draw (349,107) node [anchor=north west][inner sep=0.75pt]   [align=left] {$\alpha$};
\draw (37.33,126.33) node [anchor=north west][inner sep=0.75pt]   [align=left] {$1$};
\draw (85.33,146) node [anchor=north west][inner sep=0.75pt]   [align=left] {$2$};
\draw (126.67,153.33) node [anchor=north west][inner sep=0.75pt]   [align=left] {$3$};
\draw (198,118.33) node [anchor=north west][inner sep=0.75pt]   [align=left] {$p$};
\draw (72,40.33) node [anchor=north west][inner sep=0.75pt]   [align=left] {$p+q$};
\draw (393.33,68.67) node [anchor=north west][inner sep=0.75pt]   [align=left] {$1$};
\draw (478,139.33) node [anchor=north west][inner sep=0.75pt]   [align=left] {$2$};
\draw (378.67,175) node [anchor=north west][inner sep=0.75pt]   [align=left] {$p$};
\end{tikzpicture}
        \end{center}
            \caption{Surface models for finite-dimensional gentle one-cycle algebras.}
            \label{fig:representative-algebras}
        \end{figure}
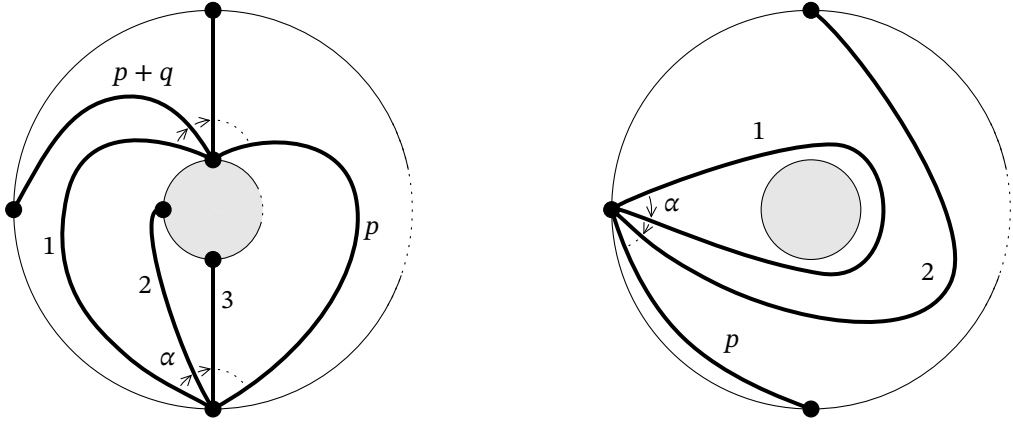
    \end{proof}

\begin{rmk}\label{rmk:KY-form}
    We note that, in the case of infinite global dimension, our choice of representative algebra is slightly different from that in~\cite{KY18}. Since for a non-positively graded gentle one-cycle algebra, the representative algebra of its derived equivalence class chosen in~\cite[Theorem~1.1(2)]{KY18} may be non-positively graded. The choice of representative in~\cite[Theorem~1.1(2)]{KY18} corresponds to the marked ribbon graph shown in Figure~\ref{fig:KY-rep}.
    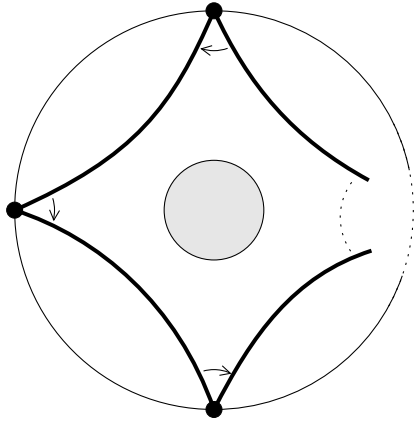
\begin{figure}
        \begin{center}
\begin{tikzpicture}[x=0.75pt,y=0.75pt,yscale=-1,xscale=1]

\draw  [draw opacity=0] (234.45,146.93) .. controls (220.84,185.98) and (183.69,214) .. (140,214) .. controls (84.77,214) and (40,169.23) .. (40,114) .. controls (40,58.77) and (84.77,14) .. (140,14) .. controls (188.37,14) and (228.72,48.34) .. (237.99,93.97) -- (140,114) -- cycle ; \draw   (234.45,146.93) .. controls (220.84,185.98) and (183.69,214) .. (140,214) .. controls (84.77,214) and (40,169.23) .. (40,114) .. controls (40,58.77) and (84.77,14) .. (140,14) .. controls (188.37,14) and (228.72,48.34) .. (237.99,93.97) ;  
\draw  [draw opacity=0][dash pattern={on 0.84pt off 2.51pt}] (230.18,70.73) .. controls (236.47,83.83) and (240,98.5) .. (240,114) .. controls (240,132.18) and (235.15,149.23) .. (226.67,163.92) -- (140,114) -- cycle ; \draw  [dash pattern={on 0.84pt off 2.51pt}] (230.18,70.73) .. controls (236.47,83.83) and (240,98.5) .. (240,114) .. controls (240,132.18) and (235.15,149.23) .. (226.67,163.92) ;  
\draw  [fill={rgb, 255:red, 0; green, 0; blue, 0 }  ,fill opacity=1 ] (136,214) .. controls (136,211.79) and (137.79,210) .. (140,210) .. controls (142.21,210) and (144,211.79) .. (144,214) .. controls (144,216.21) and (142.21,218) .. (140,218) .. controls (137.79,218) and (136,216.21) .. (136,214) -- cycle ;
\draw  [fill={rgb, 255:red, 0; green, 0; blue, 0 }  ,fill opacity=1 ] (36,114) .. controls (36,111.79) and (37.79,110) .. (40,110) .. controls (42.21,110) and (44,111.79) .. (44,114) .. controls (44,116.21) and (42.21,118) .. (40,118) .. controls (37.79,118) and (36,116.21) .. (36,114) -- cycle ;
\draw  [fill={rgb, 255:red, 0; green, 0; blue, 0 }  ,fill opacity=1 ] (136,14) .. controls (136,11.79) and (137.79,10) .. (140,10) .. controls (142.21,10) and (144,11.79) .. (144,14) .. controls (144,16.21) and (142.21,18) .. (140,18) .. controls (137.79,18) and (136,16.21) .. (136,14) -- cycle ;
\draw  [fill={rgb, 255:red, 230; green, 230; blue, 230 }  ,fill opacity=1 ] (115,114) .. controls (115,100.19) and (126.19,89) .. (140,89) .. controls (153.81,89) and (165,100.19) .. (165,114) .. controls (165,127.81) and (153.81,139) .. (140,139) .. controls (126.19,139) and (115,127.81) .. (115,114) -- cycle ;
\draw [line width=1.5]    (40,114) .. controls (101,87) and (119,62.33) .. (140,14) ;
\draw [line width=1.5]    (40,114) .. controls (82.33,127.67) and (121.67,161.67) .. (140,214) ;
\draw [line width=1.5]    (140,214) .. controls (159,177) and (179.67,147) .. (219,134.33) ;
\draw [line width=1.5]    (140,14) .. controls (157,51) and (179.67,78.33) .. (217.67,99) ;
\draw  [draw opacity=0][dash pattern={on 0.84pt off 2.51pt}] (208.56,134.59) .. controls (205.26,129.77) and (203.33,123.94) .. (203.33,117.67) .. controls (203.33,110.13) and (206.11,103.25) .. (210.7,97.98) -- (233.33,117.67) -- cycle ; \draw  [dash pattern={on 0.84pt off 2.51pt}] (208.56,134.59) .. controls (205.26,129.77) and (203.33,123.94) .. (203.33,117.67) .. controls (203.33,110.13) and (206.11,103.25) .. (210.7,97.98) ;  
\draw  [draw opacity=0] (146.84,32.8) .. controls (144.71,33.58) and (142.4,34) .. (140,34) .. controls (137.75,34) and (135.59,33.63) .. (133.58,32.95) -- (140,14) -- cycle ; \draw   (146.84,32.8) .. controls (144.71,33.58) and (142.4,34) .. (140,34) .. controls (137.75,34) and (135.59,33.63) .. (133.58,32.95) ;  
\draw  [draw opacity=0] (59.1,108.04) .. controls (59.68,109.93) and (60,111.93) .. (60,114) .. controls (60,115.7) and (59.79,117.35) .. (59.39,118.92) -- (40,114) -- cycle ; \draw   (59.1,108.04) .. controls (59.68,109.93) and (60,111.93) .. (60,114) .. controls (60,115.7) and (59.79,117.35) .. (59.39,118.92) ;  
\draw  [draw opacity=0] (134.77,194.69) .. controls (136.44,194.24) and (138.19,194) .. (140,194) .. controls (142.68,194) and (145.24,194.53) .. (147.58,195.49) -- (140,214) -- cycle ; \draw   (134.77,194.69) .. controls (136.44,194.24) and (138.19,194) .. (140,194) .. controls (142.68,194) and (145.24,194.53) .. (147.58,195.49) ;  
\draw   (62.51,114.19) -- (59.65,118.99) -- (57.52,113.82) ;
\draw   (138.18,36.41) -- (133.71,33.05) -- (139.07,31.49) ;
\draw   (144.13,192.07) -- (148.23,195.86) -- (142.74,196.87) ;
\end{tikzpicture}
        \end{center}
        \caption{Surface models of representative algebras in \cite[Theorem~1.1(2)]{KY18}}
        \label{fig:KY-rep}
    \end{figure}
\end{rmk}

Now, let $A = kQ/I$ be a graded gentle one-cycle algebra. Define $d_{+}$ (resp. $d_{-}$) as the difference between the number of clockwise (resp. counterclockwise) oriented relations and the sum of the degrees of all clockwise (resp. counterclockwise) oriented arrows.We say that $A$ satisfies the \emph{graded clock condition} if $d_{+}=d_{-}.$

In \cite[Conjecture~9.5]{KY18}, it was conjectured that derived equivalences preserve the graded clock condition. We provide a proof of this conjecture.

\begin{prop}\label{prop:gr-clock-cdt}
    Let $A = kQ/I$ be a graded gentle one-cycle algebra, and let $(\mathcal{S}_{A}, M_{A}, \eta_{A})$ be its associated graded marked annulus with $\partial\mathcal{S}_{A}=B_1\sqcup B_2$. Then the following conditions are equivalent:
\begin{enumerate}
\item $A$ satisfies the graded clock condition;
\item the grading $\eta_{A}$ is zero;
\item $A$ is derived equivalent either to the path algebra $kQ$ with the trivial grading, where $Q$ is the quiver given in Corollary~\ref{cor:standard-form}(1), or to an algebra $kQ/I$, where $Q$ is an $n$-cycle, $I$ is generated by all paths of length two, and exactly one arrow has positive degree $n$.
\end{enumerate}
\end{prop}

\begin{proof}
    Note that conditions (2) and (3) are equivalent directly by Theorem~\ref{thm:classify-graded-gentle}. 
    
     We now prove the equivalence of conditions (1) and (2). Suppose first that $A$ is infinite-dimensional. Since the arrows around the cycle correspond to the arrows around the puncture $B_i$, we have $d_-=0$ and $d_+=\eta_A(B_i)$. Therefore, in this case, conditions $(1)$ and $(2)$ are equivalent. 
     
     It remains to consider the case where $A$ is finite-dimensional.  In fact, it suffices to consider the case where the quiver of $A$ is of type $\widetilde{A}$. Indeed, on the one hand, the ``tree part'' does not affect the graded clock condition. On the other hand, if the quiver of $A$ contains a leaf, namely, an arrow which is the unique arrow starting at some source vertex or the unique arrow ending at some sink vertex, then it corresponds to a leaf in the associated ribbon graph. Such a leaf connects two consecutive marked points on a boundary component. By Definition \ref{dfn:winding-number}, removing this leaf from the ribbon graph does not change the winding numbers of the boundary components.

    Furthermore, we may assume that $e=e_1+\cdots+e_\ell$, where each idempotent $e_i$ corresponds to a vertex $i$ lying on the cycle of $Q$, and that every nonzero path $\beta\alpha\notin I$ satisfies $t(\alpha)=s(\beta)\neq i$ for all such vertices $i$. Indeed, it suffices to consider the graded gentle one-cycle algebra $eAe$. On the one hand, the graded clock condition depends only on the degrees of such nonzero paths, rather than on the degrees of the individual arrows. On the other hand, such a nonzero path  $\beta\alpha\notin I$ corresponds to a leaf in the associated ribbon graph connecting marked points on two distinct boundary components. By Definition \ref{dfn:winding-number}, deleting such a leaf also does not change the winding numbers of the boundary components.

    Finally, by the above reductions, we may assume that $A$ is a gentle one-cycle algebra whose quiver is of type $\widetilde{A}$, and that every composable path of length two belongs to the ideal $I$. In this case, the ribbon graph associated with $A$ is also a cycle without leaves. Let $\alpha_1,\ldots,\alpha_{k_1}$ be the clockwise arrows, and let $\beta_1,\ldots,\beta_{k_2}$ be the counterclockwise arrows. In the surface model, the arrows $\alpha_i$ correspond to arrows around the boundary component $B_1$, while the arrows $\beta_j$ correspond to arrows around the boundary component $B_2$. Let $r_1$ (resp. $r_2$) be the number of bigons connecting two consecutive marked points on $B_1$ (resp. $B_2$). Note that $r_1$ (resp. $r_2$) is precisely the number of clockwise (resp. counterclockwise) oriented relations. Therefore, we obtain
$$
\eta_A(B_1)=-\eta_A(B_2)=
|\beta_1|+\cdots+|\beta_{k_2}|-r_2+r_1-|\alpha_1|-\cdots-|\alpha_{k_1}|
=d_{-}-d_{+}.
$$
Hence, conditions (1) and (2) are equivalent.
\end{proof}

\subsection{Derived-discrete and silting-discrete}

In this subsection, we use the derived equivalence classification to characterize derived-discrete and silting-discrete gentle one-cycle algebras.

Recall that a locally finite connective dg algebra $A$ is called \emph{silting-discrete} if, for every $l \geq 0$, the interval between $A$ and $\Sigma^l A$ in the silting partial order (see \cite{AI12}) contains only finitely many equivalence classes of silting objects. For the precise definition and equivalent characterizations, we refer the reader to~\cite{Aih13,AM17}.

\begin{thm}\label{thm:der-dis-gentle}
    Let $A$ be a locally finite non-positively graded gentle one-cycle algebra, and let $(\mathcal{S}_{A},M_{A},\eta_{A})$ be its associated graded marked annulus. Then the following statements are equivalent:
\begin{enumerate}
\item $A$ is derived-discrete;
\item $A$ is silting-discrete;
\item the grading $\eta_A$ is nonzero;
\item $A$ is not derived equivalent to the path algebra of type $\widetilde{A}$ with the trivial grading;
\item $A$ does not satisfy the graded clock condition.
\end{enumerate}
\end{thm}

\begin{proof}
We first prove that, when $A$ is infinite-dimensional, all the above conditions hold. By Corollary~\ref{cor:standard-form} and the locally finiteness of $A$, the algebra $A$ is derived equivalent to $kQ$, where $Q$ is a graded $n$-cycle with a unique arrow of negative degree. Hence (3), (4), and (5) hold immediately, and (1) follows from Example~\ref{eg:n-cycle}. (2) then follows from (1) and \cite[Theorem 2.4]{AM17}, also follows from \cite[Corollary 4.2]{YY21}.

We now consider the case where $A$ is finite-dimensional.
Note that conditions (3) and (4) are equivalent, since $A$ is non-positively graded and by \cref{thm:classify-graded-gentle}. Moreover, conditions (3) and (5) are directly equivalent by Proposition \ref{prop:gr-clock-cdt}.

We now prove that (1) and (4) are equivalent. First, the implication $(1)\Rightarrow(4)$ is clear. Indeed, by \cref{thm:Vossieck}, the path algebra of type $\widetilde{A}$ endowed with the trivial grading is not derived-discrete. Conversely, assume that (4) holds. If $A$ has finite global dimension, then $A$ is derived-discrete by Corollary~\ref{cor:standard-form} and \cite[Theorem~A]{Fus26}.
It remains to consider the case where $A$ has infinite global dimension. Without loss of generality, we may assume that $A$ is the representative algebra up to derived equivalence given in Corollary~\ref{cor:standard-form}. Indeed, $A$ can be viewed as a subalgebra of $A'=kQ/I$, where the quiver $Q$ is of the form
$$\begin{tikzcd}
0 \arrow[r, "\alpha", bend left, shift left] & 1 \arrow[r] \arrow[l, "\beta", bend left, shift left] & 2 \arrow[r] & \cdots \arrow[r] & p
\end{tikzcd}$$
and the ideal is $I=\langle \beta\alpha\rangle$. At this point, $A'$ is also a graded gentle one-cycle algebra and the grading $\eta_{A'}=\eta_A\neq 0$. Moreover, the algebra $A'$ has finite global dimension.  Hence, by the result proved above, $A'$ is derived-discrete. Moreover, since $A=(1-e_0)A'(1-e_0)$, by Proposition \ref{prop:idempotent}, $A$ is also derived-discrete.

Finally, we prove that (2) and (3) are equivalent. If $A$ has finite global dimension, then this follows directly from \cite[Theorem~1.2]{CJSW25}; it also follows from Corollary \ref{cor:standard-form} and \cite[Theorem~A]{Fus26}. Now assume that $A$ has infinite global dimension. Without loss of generality, we also assume that $A$ is the representative algebra up to derived equivalence given in Corollary~\ref{cor:standard-form} and consider the algebra $A'$ constructed above. By the preceding argument, $A'$ is silting-discrete. Then by \cite[Theorem 1]{AH24}, $A$ is silting-discrete.
\end{proof}

\section{Semi-orthogonal decompositions}\label{sec:3}

By Proposition \ref{prop:DD-DD}, every derived-discrete locally finite non-positively graded algebra $A$ gives rise to a family of derived-discrete finite-dimensional algebras $\widetilde{A}^{[-n,0]}$ satisfying the characterization in \cref{thm:Vossieck}. Therefore, studying the shape of the derived categories of these algebras provides a way to identify the possible algebras $A$.

In this section, we use semi-orthogonal decompositions to characterize the derived category of a path algebra of type $D$.

\begin{dfn}\label{dfn:SOD}
	Let $\T$ be a triangulated category. A \emph{semi-orthogonal decomposition} of $\T$ is a pair of thick subcategories $\T_1, \T_2 \subseteq \T$ such that:
	\begin{enumerate}
		\item $\Hom_{\T}(\T_1, \T_2) = 0$,
		\item $\T$ is generated by $\T_1$ and $\T_2$, i.e., $\T = \T_1 \ast \T_2$.
	\end{enumerate}
	We write $\T = \T_1 \bot \T_2$. Note that conditions (1) and (2) force $\T_2 = \T_1^\bot$.
\end{dfn}

This section aims to prove the following theorem, which will be used in the proof of \cref{thm:main}.

\begin{thm}\label{thm:SOD}
	Let $\D^b(\mod kD_n) = \T_1 \bot \T_2$ be a semi-orthogonal decomposition. Then at least one of $\T_1, \T_2$ has no direct summand equivalent to $\D^b(\mod kD_l)$ for any $l \ge 4$.
\end{thm}

\begin{dfn}\label{dfn:wide}
	Let $\A$ be an Abelian category. A subcategory $\W \subseteq \A$ is called a \emph{wide subcategory} if $\W$ is closed under taking kernels, cokernels, and extensions.
\end{dfn}

For a hereditary abelian category $\A$, we have $\D^b(\A) = \add_{i \in \ZZ} \Sigma^i \A$. Thus, we have a bijection
\[
	\begin{array}{ccc}
		\{\text{thick subcategories of } \D^b(\A)\} & \longleftrightarrow & \{\text{wide subcategories of } \A\} \\[4pt]
		\T & \longmapsto & \T \cap \A \\[4pt]
		\thick\W & \longmapsfrom & \W
	\end{array}
\]
Since $(\thick\W)^\bot \cap \A = \W^{\bot_{0,1}}$, we have $(\thick\W)^\bot = \thick(\W^{\bot_{0,1}})$.

\begin{rmk}\label{rmk:wide-derived}
	Let $\A$ be a hereditary abelian category, and $\W \subseteq \A$ be a wide subcategory. Then $\W$ is also hereditary, and we have an equivalence $\D^b(\W) \simeq \thick_{\D^b(\A)} \W$. Indeed, since $\W$ is closed under extensions, $\Ext^1_\W \simeq \Ext^1_\A|_\W$, and the latter is right exact.
\end{rmk}

\begin{rmk}\label{rmk:wide-Morita}
	Let $Q$ be a quiver whose underlying graph is a disjoint union of Dynkin graphs. Since $\mod kQ$ is hereditary of finite representation type, every wide subcategory $\W$ of $\mod kQ$ is hereditary and admits an additive generator. By \cite[Theorem C]{CSPP22}, $\W$ therefore has a progenerator, and Morita theory gives an equivalence $\W \simeq \mod kQ_\W$ for some quiver $Q_\W$ whose underlying graph is a disjoint union of Dynkin graphs.
\end{rmk}

By Remarks \ref{rmk:wide-derived} and \ref{rmk:wide-Morita}, to prove \cref{thm:SOD} it suffices to establish the following proposition.

\begin{prop}\label{prop:typeD}
	Let $\W \simeq \mod kQ_\W \subseteq \mod kD_n$ be a wide subcategory. If $Q_\W$ has a connected component of type $D_k$ for some $k \ge 4$, then $Q_{\W^{\bot_{0,1}}}$ has no connected component of type $D_l$ for any $l \ge 4$.
\end{prop}

We give a proof of Proposition \ref{prop:typeD} at the end of \cref{subsection:NP type D}. Granting it, we can prove \cref{thm:SOD}.

\begin{proof}[Proof of \cref{thm:SOD}]
	Write $\W_1 := \T_1 \cap \mod kD_n$, so that $\T_1 = \thick \W_1$ and, since $\T_2 = \T_1^\bot$, also $\T_2 = \thick(\W_1^{\bot_{0,1}})$. By Remarks \ref{rmk:wide-derived} and \ref{rmk:wide-Morita}, $\T_1$ (resp.\ $\T_2$) admits $\D^b(\mod kD_l)$ as a direct summand if and only if $Q_{\W_1}$ (resp.\ $Q_{\W_1^{\bot_{0,1}}}$) has a connected component of type $D_l$. By Proposition \ref{prop:typeD}, these two quivers cannot both have a connected component of type $D_l$ for some $l \ge 4$, so at least one of $\T_1, \T_2$ has no such summand.
\end{proof}

\subsection{Noncrossing partitions and wide subcategories}

In this subsection, we recall the results in \cite{IT09}.

Let $Q$ be a Dynkin quiver. The \emph{Euler form} $\la-,-\ra$ on $K_0(\mod kQ)$ is defined by
\begin{align*}
    \la[M],[N]\ra:=\dim\Hom_{kQ}(M,N)-\dim\Ext^1_{kQ}(M,N).
\end{align*}
We put $([M],[N]):=\la [M],[N]\ra+\la[N],[M]\ra$. For $M\in\ind\left(\mod kQ\right)$, we put
\begin{align*}
    s_{[M]}([N]):=[N]-([M],[N])[M].
\end{align*}
The \emph{Weyl group} $W=W_Q$ of $Q$ is a subgroup of $\Aut(K_0(\mod kQ))$ generated by $\{s_{[M]}\mid M\in\ind\mod kQ\}$. For $M$ indecomposable, $s_{[M]}$ is called a \emph{reflection}.

\begin{dfn}[Absolute order]
    For $w\in W$, let $\l(w)$ be the length of the shortest expression for $w$ as a product of reflections. For $u,v\in W$, let $u\le v$ if
    \begin{align*}
        \l(v)=\l(u)+\l(u^{-1}v)
    \end{align*}
    holds.
\end{dfn}

\begin{thm}\label{thm:IT}\cite[Theorem 1.1, Theorem 3.13]{IT09}
    Let $Q$ be a Dynkin quiver, and $W$ be its Weyl group.For a wide subcategory $\W$, let $M_1,\ldots,M_r$ be representatives of the isomorphism classes of simple objects in $\W$, labeled so that $\Ext^1(M_i,M_j)=0$ whenever $i<j$. We put $\cox(\W):=s_{M_1}\cdots s_{M_r}\in W$. Then we obtain a poset isomorphism
    \[
    \begin{array}{ccc}
        \wide(\mod kQ):=\{\W\subseteq\mod kQ\mid \W \text{ is a wide subcategory}\} & \longleftrightarrow & \NC_Q:=\{w\in W\mid e\le w\le c\} \\[4pt]
        \W & \longmapsto & \cox(\W) \\[4pt]
        \W_w:=\add\{M\in\ind\mod kQ\mid s_M\le w \} & \longmapsfrom & w
    \end{array}
    \]
    where $c=c_Q:=\cox(\mod kQ)$ is the \emph{Coxeter element} of $Q$. In addition, $\cox(\W)\cox(\W^{\bot_{0,1}})=c$ holds for every wide subcategory $\W$. 

    Here, the elements of $\NC_Q$ are called \emph{noncrossing partitions} of $W$.
\end{thm}

\begin{dfn}
    For a noncrossing partition $w\in\NC_Q$, we call $\Kr(w):=w^{-1}c$ the \emph{Kreweras complement} of $w$.
\end{dfn}

The last statement of \cref{thm:IT} says that $\W_w^{\bot_{0,1}}=\W_{\Kr(w)}$.

\subsection{Noncrossing partition of type $D_n$}\label{subsection:NP type D}

Let $\overrightarrow{D}_n$ be the following quiver:
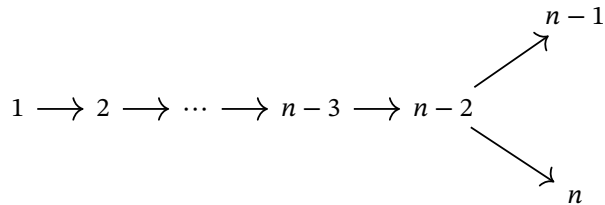
\begin{center}
    \begin{tikzcd}[row sep = 2em, column sep = 1.8em]
        & & & & & {n-1} \\
        1 \arrow[r]
        & 2 \arrow[r]
        & \cdots \arrow[r]
        & {n-3} \arrow[r]
        & {n-2} \arrow[ur] \arrow[dr]
        & \\
        & & & & & {n}
    \end{tikzcd}
\end{center}

The Weyl group $W=W_{\ora{D}_n}$ is realized as a subgroup of $S_{[\pm n]}$, the symmetric group on $[\pm n]:=\{1,\ldots,n,-1,\ldots,-n\}$. For $i\neq -j$, define the paired transposition
\[
    ((i,j)):=(i,j)(-i,-j)\in S_{[\pm n]},
\]
which simultaneously swaps $i\leftrightarrow j$ and $-i\leftrightarrow -j$. The simple reflections are identified as
\[
    s_i=((i,\,i+1))\quad(1\le i\le n-1),
    \qquad
    s_n=((-(n-1),\,n)),
\]
so $W$ is the group of even signed permutations: those $\sigma\in S_{[\pm n]}$ with $\sigma(-i)=-\sigma(i)$ for all $i$ and an even number of sign changes among $\sigma(1),\ldots,\sigma(n)$.

Every $w\in W$ decomposes uniquely into disjoint cycles of two types:
\begin{itemize}
    \item \emph{Paired cycle}: $((i_1,\ldots,i_k)):=c\bar c$, where
          $c=(i_1,\ldots,i_k)$ and $\bar c=(-i_1,\ldots,-i_k)$ are disjoint.
    \item \emph{Balanced cycle}: $[i_1,\ldots,i_k]:=(i_1,\ldots,i_k,-i_1,\ldots,-i_k)$.
\end{itemize}
The Coxeter element $c=\cox(\mod k\ora{D}_n)$ is
\begin{align*}
    c = s_1\cdots s_{n-1}s_n
      = ((1,2))\cdots((n-1,n))\cdot((-(n-1),n))
      = [1,\ldots,n-1]\,[n].
\end{align*}

The cycle structure of $w\in W$ naturally encodes a partition of $[\pm n]$:
\begin{itemize}
    \item each paired cycle $((i_1,\ldots,i_k))$ yields a pair of blocks $\{i_1,\ldots,i_k\}$ and $\{-i_1,\ldots,-i_k\}$,
    \item each balanced cycle $[i_1,\ldots,i_k]$ yields a single block $\{i_1,\ldots,i_k,-i_1,\ldots,-i_k\}$ fixed by negation.
\end{itemize}
This motivates the following definition.

\begin{dfn}
    A \emph{$D_n$-partition} is a partition $\pi$ of $[\pm n]$ satisfying:
    \begin{itemize}
        \item $B\in\pi \Rightarrow -B\in\pi$,
        \item at most one block is fixed by negation (\emph{zero block}),
        \item if a zero block is present, it contains more than two elements.
    \end{itemize}
    Let $\Pi^D(n)$ denote the set of all $D_n$-partitions.
\end{dfn}

Athanasiadis--Reiner \cite{AR04} gives a visual noncrossing condition for $D_n$-partitions. Label the vertices of a regular $(2n-2)$-gon clockwise by
$1,2,\ldots,n-1,-1,\ldots,-(n-1)$, and place both $n$ and $-n$ at the centroid. For a block $B$, let $\rho(B)$ be the convex hull of the points labeled by $B$. Two distinct blocks $B,B'$ \emph{cross} if $\rho(B)\ne\rho(B')$ and one contains a point of the other in its relative interior.

\begin{dfn}
    Let $\NC^D(n)\subseteq\Pi^D(n)$ be the subposet consisting of noncrossing
    $D_n$-partitions.
\end{dfn}

\begin{eg}
	We give some examples of $D_4$-noncrossing partitions:
\begin{center}
\begin{minipage}{0.30\textwidth}\centering
	\begin{tikzpicture}[baseline=(vc.base)]
		\def\R{1.35cm}
		\coordinate (v1) at (90:\R);
		\coordinate (v2) at (30:\R);
		\coordinate (v3) at (-30:\R);
		\coordinate (v4) at (-90:\R);
		\coordinate (v5) at (-150:\R);
		\coordinate (v6) at (150:\R);
		\draw[gray!40,thin] (v1)--(v2)--(v3)--(v4)--(v5)--(v6)--cycle;
		\node[font=\tiny] at ($(v1)+(0,0.28)$)      {$1$};
		\node[font=\tiny] at ($(v2)+(0.26,0.15)$)   {$2$};
		\node[font=\tiny] at ($(v3)+(0.26,-0.15)$)  {$3$};
		\node[font=\tiny] at ($(v4)+(0,-0.28)$)      {$-1$};
		\node[font=\tiny] at ($(v5)+(-0.26,-0.15)$) {$-2$};
		\node[font=\tiny] at ($(v6)+(-0.26,0.15)$)  {$-3$};
		\node[font=\tiny,black] at (0.22,-0.21)   {$\pm4$};
		\foreach \v in {v1,v2,v3,v4,v5,v6}{
			\node[circle,draw=black,fill=white,inner sep=0pt,
				minimum size=5pt,line width=0.85pt] at (\v) {};}
		\node[circle,draw=black!75,fill=white,inner sep=0pt,
			minimum size=5pt,line width=0.85pt] (vc) at (0,0) {};
	\end{tikzpicture}
	\par\smallskip{\footnotesize $\{1\},\{-1\},\ldots,\{4\},\{-4\}$ (min)}
\end{minipage}%
%
\begin{minipage}{0.30\textwidth}\centering
	\begin{tikzpicture}[baseline=(vc.base)]
		\def\R{1.35cm}
		\coordinate (v1) at (90:\R);
		\coordinate (v2) at (30:\R);
		\coordinate (v3) at (-30:\R);
		\coordinate (v4) at (-90:\R);
		\coordinate (v5) at (-150:\R);
		\coordinate (v6) at (150:\R);
		\draw[gray!40,thin] (v1)--(v2)--(v3)--(v4)--(v5)--(v6)--cycle;
		\fill[gray!25,opacity=0.85] (v1)--(v2)--(v3)--cycle;
		\draw[black!75,line width=1.7pt,line join=round] (v1)--(v2)--(v3)--cycle;
		\fill[gray!25,opacity=0.85] (v4)--(v5)--(v6)--cycle;
		\draw[black!75,line width=1.7pt,line join=round] (v4)--(v5)--(v6)--cycle;
		\node[font=\tiny] at ($(v1)+(0,0.28)$)      {$1$};
		\node[font=\tiny] at ($(v2)+(0.26,0.15)$)   {$2$};
		\node[font=\tiny] at ($(v3)+(0.26,-0.15)$)  {$3$};
		\node[font=\tiny] at ($(v4)+(0,-0.28)$)      {$-1$};
		\node[font=\tiny] at ($(v5)+(-0.26,-0.15)$) {$-2$};
		\node[font=\tiny] at ($(v6)+(-0.26,0.15)$)  {$-3$};
		\node[font=\tiny,black] at (0.22,-0.21)   {$\pm4$};
		\foreach \v in {v1,v2,v3,v4,v5,v6}{
			\node[circle,draw=black,fill=white,inner sep=0pt,
				minimum size=5pt,line width=0.85pt] at (\v) {};}
		\node[circle,draw=black!75,fill=white,inner sep=0pt,
			minimum size=5pt,line width=0.85pt] (vc) at (0,0) {};
	\end{tikzpicture}
	\par\smallskip{\footnotesize $\{1,2,3\},\{-1,-2,-3\},\{4\},\{-4\}$}
\end{minipage}%
%
\begin{minipage}{0.30\textwidth}\centering
	\begin{tikzpicture}[baseline=(vc.base)]
		\def\R{1.35cm}
		\coordinate (v1) at (90:\R);
		\coordinate (v2) at (30:\R);
		\coordinate (v3) at (-30:\R);
		\coordinate (v4) at (-90:\R);
		\coordinate (v5) at (-150:\R);
		\coordinate (v6) at (150:\R);
		\draw[gray!40,thin] (v1)--(v2)--(v3)--(v4)--(v5)--(v6)--cycle;
		\fill[gray!25,opacity=0.85] (0,0)--(v1)--(v2)--cycle;
		\draw[black!75,line width=1.7pt,line join=round] (0,0)--(v1)--(v2)--cycle;
		\fill[gray!25,opacity=0.85] (0,0)--(v4)--(v5)--cycle;
		\draw[black!75,line width=1.7pt,line join=round] (0,0)--(v4)--(v5)--cycle;
		\node[font=\tiny] at ($(v1)+(0,0.28)$)      {$1$};
		\node[font=\tiny] at ($(v2)+(0.26,0.15)$)   {$2$};
		\node[font=\tiny] at ($(v3)+(0.26,-0.15)$)  {$3$};
		\node[font=\tiny] at ($(v4)+(0,-0.28)$)      {$-1$};
		\node[font=\tiny] at ($(v5)+(-0.26,-0.15)$) {$-2$};
		\node[font=\tiny] at ($(v6)+(-0.26,0.15)$)  {$-3$};
		\node[font=\tiny,black] at (0.30, 0.38)   {$+4$};
		\node[font=\tiny,black] at (-0.30,-0.38)  {$-4$};
		\foreach \v in {v1,v2,v3,v4,v5,v6}{
			\node[circle,draw=black,fill=white,inner sep=0pt,
				minimum size=5pt,line width=0.85pt] at (\v) {};}
		\node[circle,draw=black!75,fill=white,inner sep=0pt,
			minimum size=5pt,line width=0.85pt] (vc) at (0,0) {};
	\end{tikzpicture}
	\par\smallskip{\footnotesize $\{1,2,4\},\{-1,-2,-4\},\{3\},\{-3\}$}
\end{minipage}

\medskip
\begin{minipage}{0.30\textwidth}\centering
	\begin{tikzpicture}[baseline=(vc.base)]
		\def\R{1.35cm}
		\coordinate (v1) at (90:\R);
		\coordinate (v2) at (30:\R);
		\coordinate (v3) at (-30:\R);
		\coordinate (v4) at (-90:\R);
		\coordinate (v5) at (-150:\R);
		\coordinate (v6) at (150:\R);
		\draw[gray!40,thin] (v1)--(v2)--(v3)--(v4)--(v5)--(v6)--cycle;
		\draw[black!75,line width=2.0pt,line cap=round] (v1)--(v4);
		\draw[black!75,line width=1.7pt,line cap=round] (v2)--(v3);
		\draw[black!75,line width=1.7pt,line cap=round] (v5)--(v6);
		\node[font=\tiny] at ($(v1)+(0,0.28)$)      {$1$};
		\node[font=\tiny] at ($(v2)+(0.26,0.15)$)   {$2$};
		\node[font=\tiny] at ($(v3)+(0.26,-0.15)$)  {$3$};
		\node[font=\tiny] at ($(v4)+(0,-0.28)$)      {$-1$};
		\node[font=\tiny] at ($(v5)+(-0.26,-0.15)$) {$-2$};
		\node[font=\tiny] at ($(v6)+(-0.26,0.15)$)  {$-3$};
		\node[font=\tiny,black] at (0.22,-0.21)   {$\pm4$};
		\foreach \v in {v1,v2,v3,v4,v5,v6}{
			\node[circle,draw=black,fill=white,inner sep=0pt,
				minimum size=5pt,line width=0.85pt] at (\v) {};}
		\node[circle,draw=black!75,fill=white,inner sep=0pt,
			minimum size=5pt,line width=0.85pt] (vc) at (0,0) {};
	\end{tikzpicture}
	\par\smallskip{\footnotesize $\{1,-1,4,-4\},\;\{2,3\},\{-2,-3\}$}
\end{minipage}%
%
\begin{minipage}{0.30\textwidth}\centering
	\begin{tikzpicture}[baseline=(vc.base)]
		\def\R{1.35cm}
		\coordinate (v1) at (90:\R);
		\coordinate (v2) at (30:\R);
		\coordinate (v3) at (-30:\R);
		\coordinate (v4) at (-90:\R);
		\coordinate (v5) at (-150:\R);
		\coordinate (v6) at (150:\R);
		\draw[gray!40,thin] (v1)--(v2)--(v3)--(v4)--(v5)--(v6)--cycle;
		\fill[gray!25,opacity=0.90] (v1)--(v2)--(v4)--(v5)--cycle;
		\draw[black!75,line width=1.7pt,line join=round]
			(v1)--(v2)--(v4)--(v5)--cycle;
		\node[font=\tiny] at ($(v1)+(0,0.28)$)      {$1$};
		\node[font=\tiny] at ($(v2)+(0.26,0.15)$)   {$2$};
		\node[font=\tiny] at ($(v3)+(0.26,-0.15)$)  {$3$};
		\node[font=\tiny] at ($(v4)+(0,-0.28)$)      {$-1$};
		\node[font=\tiny] at ($(v5)+(-0.26,-0.15)$) {$-2$};
		\node[font=\tiny] at ($(v6)+(-0.26,0.15)$)  {$-3$};
		\node[font=\tiny,black] at (0.22,-0.21)   {$\pm4$};
		\foreach \v in {v1,v2,v3,v4,v5,v6}{
			\node[circle,draw=black,fill=white,inner sep=0pt,
				minimum size=5pt,line width=0.85pt] at (\v) {};}
		\node[circle,draw=black!75,fill=white,inner sep=0pt,
			minimum size=5pt,line width=0.85pt] (vc) at (0,0) {};
	\end{tikzpicture}
	\par\smallskip{\footnotesize $\{1,-1,2,-2,4,-4\},\{3\},\{-3\}$}
\end{minipage}%
%
\begin{minipage}{0.30\textwidth}\centering
	\begin{tikzpicture}[baseline=(vc.base)]
		\def\R{1.35cm}
		\coordinate (v1) at (90:\R);
		\coordinate (v2) at (30:\R);
		\coordinate (v3) at (-30:\R);
		\coordinate (v4) at (-90:\R);
		\coordinate (v5) at (-150:\R);
		\coordinate (v6) at (150:\R);
		\draw[gray!40,thin] (v1)--(v2)--(v3)--(v4)--(v5)--(v6)--cycle;
		\fill[gray!25,opacity=0.90]
			(v1)--(v2)--(v3)--(v4)--(v5)--(v6)--cycle;
		\draw[black!75,line width=1.7pt,line join=round]
			(v1)--(v2)--(v3)--(v4)--(v5)--(v6)--cycle;
		\node[font=\tiny] at ($(v1)+(0,0.28)$)      {$1$};
		\node[font=\tiny] at ($(v2)+(0.26,0.15)$)   {$2$};
		\node[font=\tiny] at ($(v3)+(0.26,-0.15)$)  {$3$};
		\node[font=\tiny] at ($(v4)+(0,-0.28)$)      {$-1$};
		\node[font=\tiny] at ($(v5)+(-0.26,-0.15)$) {$-2$};
		\node[font=\tiny] at ($(v6)+(-0.26,0.15)$)  {$-3$};
		\node[font=\tiny,black] at (0.22,-0.21)   {$\pm4$};
		\foreach \v in {v1,v2,v3,v4,v5,v6}{
			\node[circle,draw=black,fill=white,inner sep=0pt,
				minimum size=5pt,line width=0.85pt] at (\v) {};}
		\node[circle,draw=black!75,fill=white,inner sep=0pt,
			minimum size=5pt,line width=0.85pt] (vc) at (0,0) {};
	\end{tikzpicture}
	\par\smallskip{\footnotesize $[\pm4]$ (max)}
\end{minipage}
\end{center}
\end{eg}

\begin{thm}\cite[Theorem 1.1]{AR04}\label{thm:AR-iso}
	The map $\NC^D(n)\to\NC_{\ora{D}_n}$ sending $\pi$ to the element $w\in W$ whose
	\begin{itemize}
		\item paired cycles are $((i_1,\ldots,i_k))$ for each pair of nonzero blocks $\{i_1,\ldots,i_k\},\{-i_1,\ldots,-i_k\}$ of $\pi$, with $i_1,\ldots,i_k$ ordered by the cyclic order $-1,-2,\ldots,-n,1,2,\ldots,n$, and
		\item balanced cycles are $[n]$ and the cycle of $Z\setminus\{n,-n\}$ in the same order, if $\pi$ has a zero block $Z$,
	\end{itemize}
	is a poset isomorphism.
\end{thm}

\begin{dfn}
    For a noncrossing $D_n$-partition $\pi\in\NC^D(n)$, we define its
    \emph{Kreweras complement} $\Kr(\pi)$ via the poset isomorphism
    $\NC^D(n)\simeq\NC_{\ora{D}_n}$ of
    \cref{thm:AR-iso}.
\end{dfn}

Combined with \cref{thm:IT}, \cref{thm:AR-iso} gives the following.

\begin{thm}
    There is a poset isomorphism
    \[
        \W_{(-)}\colon \NC^D(n)\simeq\wide(\mod k\ora{D}_n).
    \]
    Moreover, $\W_\pi^{\bot_{0,1}}=\W_{\Kr(\pi)}$ for every $\pi\in\NC^D(n)$.
\end{thm}

For $\pi\in\NC^D(n)$, the above isomorphism is restricted to the poset isomorphism
\begin{align*}
    \{\pi'\in\NC^D(n)\mid \pi'\subseteq \pi\}\simeq\wide(\W_\pi).
\end{align*}

The following proposition is an immediate consequence of the above observation.

\begin{prop}\label{prop:str}
    Let $\pi$ be a noncrossing partition. Then 
    \begin{align*}
        \W_\pi\simeq
        \begin{cases}
            \bigoplus_{k\ge1}\mod kA_{k-1}^{\#\{\text{nonzero block of size $k$}\}/\pm1}&\text{ if there is no zero block} \\
            \bigoplus_{k\ge1}\mod kA_{k-1}^{\#\{\text{nonzero block of size $k$}\}/\pm1}\oplus\mod kD_l&\text{ if there is a zero block of size $2l$}
        \end{cases}
    \end{align*}
\end{prop}
\begin{proof}
    $W_\pi$ is equivalent to the direct sum of the module categories of Dynkin quivers. The Dynkin type is recovered from the poset of wide subcategories. By the unique factorization theorem of posets \cite{H51}, we obtain the desired equivalence.
\end{proof}

The following lemma follows directly from the explicit description of the isomorphism in \cref{thm:AR-iso}.
\begin{lem}\label{lem:Kreweras complement}
    Let $\pi$ be a noncrossing $D_n$-partition. If $\pi$ contains a zero block, then we can compute $\Kr(\pi)$ as follows:
    \begin{itemize}
        \item [(i)] Let $\pi'$ be a slight counterclockwise rotation of $\pi$,
        \item [(ii)] then $\Kr(\pi)$ is the maximal noncrossing partition satisfying $\pi'\cap\Kr(\pi
        )=\emptyset$.
    \end{itemize}
\end{lem}

\begin{eg}
	We illustrate the Kreweras complement with three examples in $\NC^D(4)$.
\begin{center}
\begin{minipage}[c]{0.42\textwidth}\centering
	\begin{tikzpicture}[baseline=(vc.base)]
		\def\R{1.35cm}
		\coordinate (v1) at (90:\R);
		\coordinate (v2) at (30:\R);
		\coordinate (v3) at (-30:\R);
		\coordinate (v4) at (-90:\R);
		\coordinate (v5) at (-150:\R);
		\coordinate (v6) at (150:\R);
		\draw[gray!40,thin] (v1)--(v2)--(v3)--(v4)--(v5)--(v6)--cycle;
		\draw[black!75,line width=2.0pt,line cap=round] (v1)--(v4);
		\node[font=\tiny] at ($(v1)+(0,0.28)$)      {$1$};
		\node[font=\tiny] at ($(v2)+(0.26,0.15)$)   {$2$};
		\node[font=\tiny] at ($(v3)+(0.26,-0.15)$)  {$3$};
		\node[font=\tiny] at ($(v4)+(0,-0.28)$)      {$-1$};
		\node[font=\tiny] at ($(v5)+(-0.26,-0.15)$) {$-2$};
		\node[font=\tiny] at ($(v6)+(-0.26,0.15)$)  {$-3$};
		\node[font=\tiny,black] at (0.22,-0.21)   {$\pm4$};
		\foreach \v in {v1,v2,v3,v4,v5,v6}{
			\node[circle,draw=black,fill=white,inner sep=0pt,
				minimum size=5pt,line width=0.85pt] at (\v) {};}
		\node[circle,draw=black!75,fill=white,inner sep=0pt,
			minimum size=5pt,line width=0.85pt] (vc) at (0,0) {};
	\end{tikzpicture}
	\par\smallskip{\footnotesize $\{1,-1,4,-4\},\{2\},\{-2\},\{3\},\{-3\}$}
\end{minipage}%
\begin{minipage}[c]{0.08\textwidth}\centering
	$\xrightarrow{\smash{\Kr}}$
\end{minipage}%
\begin{minipage}[c]{0.42\textwidth}\centering
	\begin{tikzpicture}[baseline=(vc.base)]
		\def\R{1.35cm}
		\coordinate (v1) at (90:\R);
		\coordinate (v2) at (30:\R);
		\coordinate (v3) at (-30:\R);
		\coordinate (v4) at (-90:\R);
		\coordinate (v5) at (-150:\R);
		\coordinate (v6) at (150:\R);
		\draw[gray!40,thin] (v1)--(v2)--(v3)--(v4)--(v5)--(v6)--cycle;
		\fill[gray!25,opacity=0.85] (v1)--(v2)--(v3)--cycle;
		\draw[black!75,line width=1.7pt,line join=round] (v1)--(v2)--(v3)--cycle;
		\fill[gray!25,opacity=0.85] (v4)--(v5)--(v6)--cycle;
		\draw[black!75,line width=1.7pt,line join=round] (v4)--(v5)--(v6)--cycle;
		\node[font=\tiny] at ($(v1)+(0,0.28)$)      {$1$};
		\node[font=\tiny] at ($(v2)+(0.26,0.15)$)   {$2$};
		\node[font=\tiny] at ($(v3)+(0.26,-0.15)$)  {$3$};
		\node[font=\tiny] at ($(v4)+(0,-0.28)$)      {$-1$};
		\node[font=\tiny] at ($(v5)+(-0.26,-0.15)$) {$-2$};
		\node[font=\tiny] at ($(v6)+(-0.26,0.15)$)  {$-3$};
		\node[font=\tiny,black] at (0.22,-0.21)   {$\pm4$};
		\foreach \v in {v1,v2,v3,v4,v5,v6}{
			\node[circle,draw=black,fill=white,inner sep=0pt,
				minimum size=5pt,line width=0.85pt] at (\v) {};}
		\node[circle,draw=black!75,fill=white,inner sep=0pt,
			minimum size=5pt,line width=0.85pt] (vc) at (0,0) {};
	\end{tikzpicture}
	\par\smallskip{\footnotesize $\{1,2,3\},\{-1,-2,-3\},\{4\},\{-4\}$}
\end{minipage}

\medskip

\begin{minipage}[c]{0.42\textwidth}\centering
	\begin{tikzpicture}[baseline=(vc.base)]
		\def\R{1.35cm}
		\coordinate (v1) at (90:\R);
		\coordinate (v2) at (30:\R);
		\coordinate (v3) at (-30:\R);
		\coordinate (v4) at (-90:\R);
		\coordinate (v5) at (-150:\R);
		\coordinate (v6) at (150:\R);
		\draw[gray!40,thin] (v1)--(v2)--(v3)--(v4)--(v5)--(v6)--cycle;
		\draw[black!75,line width=2.0pt,line cap=round] (v1)--(v4);
		\draw[black!75,line width=1.7pt,line cap=round] (v2)--(v3);
		\draw[black!75,line width=1.7pt,line cap=round] (v5)--(v6);
		\node[font=\tiny] at ($(v1)+(0,0.28)$)      {$1$};
		\node[font=\tiny] at ($(v2)+(0.26,0.15)$)   {$2$};
		\node[font=\tiny] at ($(v3)+(0.26,-0.15)$)  {$3$};
		\node[font=\tiny] at ($(v4)+(0,-0.28)$)      {$-1$};
		\node[font=\tiny] at ($(v5)+(-0.26,-0.15)$) {$-2$};
		\node[font=\tiny] at ($(v6)+(-0.26,0.15)$)  {$-3$};
		\node[font=\tiny,black] at (0.22,-0.21)   {$\pm4$};
		\foreach \v in {v1,v2,v3,v4,v5,v6}{
			\node[circle,draw=black,fill=white,inner sep=0pt,
				minimum size=5pt,line width=0.85pt] at (\v) {};}
		\node[circle,draw=black!75,fill=white,inner sep=0pt,
			minimum size=5pt,line width=0.85pt] (vc) at (0,0) {};
	\end{tikzpicture}
	\par\smallskip{\footnotesize $\{1,-1,4,-4\},\{2,3\},\{-2,-3\}$}
\end{minipage}%
\begin{minipage}[c]{0.08\textwidth}\centering
	$\xrightarrow{\smash{\Kr}}$
\end{minipage}%
\begin{minipage}[c]{0.42\textwidth}\centering
	\begin{tikzpicture}[baseline=(vc.base)]
		\def\R{1.35cm}
		\coordinate (v1) at (90:\R);
		\coordinate (v2) at (30:\R);
		\coordinate (v3) at (-30:\R);
		\coordinate (v4) at (-90:\R);
		\coordinate (v5) at (-150:\R);
		\coordinate (v6) at (150:\R);
		\draw[gray!40,thin] (v1)--(v2)--(v3)--(v4)--(v5)--(v6)--cycle;
		\draw[black!75,line width=1.7pt,line cap=round] (v1)--(v3);
		\draw[black!75,line width=1.7pt,line cap=round] (v4)--(v6);
		\node[font=\tiny] at ($(v1)+(0,0.28)$)      {$1$};
		\node[font=\tiny] at ($(v2)+(0.26,0.15)$)   {$2$};
		\node[font=\tiny] at ($(v3)+(0.26,-0.15)$)  {$3$};
		\node[font=\tiny] at ($(v4)+(0,-0.28)$)      {$-1$};
		\node[font=\tiny] at ($(v5)+(-0.26,-0.15)$) {$-2$};
		\node[font=\tiny] at ($(v6)+(-0.26,0.15)$)  {$-3$};
		\node[font=\tiny,black] at (0.22,-0.21)   {$\pm4$};
		\foreach \v in {v1,v2,v3,v4,v5,v6}{
			\node[circle,draw=black,fill=white,inner sep=0pt,
				minimum size=5pt,line width=0.85pt] at (\v) {};}
		\node[circle,draw=black!75,fill=white,inner sep=0pt,
			minimum size=5pt,line width=0.85pt] (vc) at (0,0) {};
	\end{tikzpicture}
	\par\smallskip{\footnotesize $\{1,3\},\{-1,-3\},\{2\},\{-2\},\{4\},\{-4\}$}
\end{minipage}

\medskip

\begin{minipage}[c]{0.42\textwidth}\centering
	\begin{tikzpicture}[baseline=(vc.base)]
		\def\R{1.35cm}
		\coordinate (v1) at (90:\R);
		\coordinate (v2) at (30:\R);
		\coordinate (v3) at (-30:\R);
		\coordinate (v4) at (-90:\R);
		\coordinate (v5) at (-150:\R);
		\coordinate (v6) at (150:\R);
		\draw[gray!40,thin] (v1)--(v2)--(v3)--(v4)--(v5)--(v6)--cycle;
		\fill[gray!25,opacity=0.90] (v1)--(v2)--(v4)--(v5)--cycle;
		\draw[black!75,line width=1.7pt,line join=round]
			(v1)--(v2)--(v4)--(v5)--cycle;
		\node[font=\tiny] at ($(v1)+(0,0.28)$)      {$1$};
		\node[font=\tiny] at ($(v2)+(0.26,0.15)$)   {$2$};
		\node[font=\tiny] at ($(v3)+(0.26,-0.15)$)  {$3$};
		\node[font=\tiny] at ($(v4)+(0,-0.28)$)      {$-1$};
		\node[font=\tiny] at ($(v5)+(-0.26,-0.15)$) {$-2$};
		\node[font=\tiny] at ($(v6)+(-0.26,0.15)$)  {$-3$};
		\node[font=\tiny,black] at (0.22,-0.21)   {$\pm4$};
		\foreach \v in {v1,v2,v3,v4,v5,v6}{
			\node[circle,draw=black,fill=white,inner sep=0pt,
				minimum size=5pt,line width=0.85pt] at (\v) {};}
		\node[circle,draw=black!75,fill=white,inner sep=0pt,
			minimum size=5pt,line width=0.85pt] (vc) at (0,0) {};
	\end{tikzpicture}
	\par\smallskip{\footnotesize $\{1,-1,2,-2,4,-4\},\{3\},\{-3\}$}
\end{minipage}%
\begin{minipage}[c]{0.08\textwidth}\centering
	$\xrightarrow{\smash{\Kr}}$
\end{minipage}%
\begin{minipage}[c]{0.42\textwidth}\centering
	\begin{tikzpicture}[baseline=(vc.base)]
		\def\R{1.35cm}
		\coordinate (v1) at (90:\R);
		\coordinate (v2) at (30:\R);
		\coordinate (v3) at (-30:\R);
		\coordinate (v4) at (-90:\R);
		\coordinate (v5) at (-150:\R);
		\coordinate (v6) at (150:\R);
		\draw[gray!40,thin] (v1)--(v2)--(v3)--(v4)--(v5)--(v6)--cycle;
		\draw[black!75,line width=1.7pt,line cap=round] (v2)--(v3);
		\draw[black!75,line width=1.7pt,line cap=round] (v5)--(v6);
		\node[font=\tiny] at ($(v1)+(0,0.28)$)      {$1$};
		\node[font=\tiny] at ($(v2)+(0.26,0.15)$)   {$2$};
		\node[font=\tiny] at ($(v3)+(0.26,-0.15)$)  {$3$};
		\node[font=\tiny] at ($(v4)+(0,-0.28)$)      {$-1$};
		\node[font=\tiny] at ($(v5)+(-0.26,-0.15)$) {$-2$};
		\node[font=\tiny] at ($(v6)+(-0.26,0.15)$)  {$-3$};
		\node[font=\tiny,black] at (0.22,-0.21)   {$\pm4$};
		\foreach \v in {v1,v2,v3,v4,v5,v6}{
			\node[circle,draw=black,fill=white,inner sep=0pt,
				minimum size=5pt,line width=0.85pt] at (\v) {};}
		\node[circle,draw=black!75,fill=white,inner sep=0pt,
			minimum size=5pt,line width=0.85pt] (vc) at (0,0) {};
	\end{tikzpicture}
	\par\smallskip{\footnotesize $\{2,3\},\{-2,-3\},\{1\},\{-1\},\{4\},\{-4\}$}
\end{minipage}
\end{center}
\end{eg}

Now we are able to prove that if a wide subcategory $\W$ contains $\mod kD_k$ as a component for some $k\ge4$, then $\W^{\bot_{0,1}}$ does not contain $\mod kD_l$ as a component for any $l\ge4$.

\begin{proof}[Proof of Proposition \ref{prop:typeD}]
    Let $\W$ be a wide subcategory of $\mod k\ora{D}_n$ containing $\mod kD_k$ as a component for some $k\ge4$. 
    Let $\pi$ be a noncrossing $D_n$-partition such that $\W_\pi=\W$. By Proposition \ref{prop:str}, the partition $\pi$ contains a zero block.
    By Lemma \ref{lem:Kreweras complement}, if $\pi$ contains a zero block, then $\Kr(\pi)$ does not. Proposition \ref{prop:str} implies that $\W_\pi^{\bot_{0,1}}=\W_{\Kr(\pi)}$ does not contain $\mod kD_l$ as a component for any $l\ge4$.
\end{proof}

\section{Main result}\label{sec:5}

Recall that two graded algebras $A$ and $B$ are said to be \emph{graded Morita equivalent} if there is an equivalence between the categories of graded modules $\Grmod A$ and $\Grmod B$ that is compatible with the grading shifts.

We are now ready to prove our main theorem.

\begin{thm}\label{thm:main}
    Let $A=kQ/I$ be a locally finite non-positively graded algebra. Then, the following conditions are equivalent:
    \begin{itemize}
        \item [(1)] $A$ is derived-discrete,
        \item [(2)] $A$ is either 
        \begin{itemize}
            \item [(a)] graded Morita equivalent to a piecewise hereditary algebra of Dynkin type,
            \item [(b)] a graded gentle one-cycle algebra not satisfying the graded clock condition.
        \end{itemize}
    \end{itemize}
\end{thm}
\begin{proof}
    $(2)\To(1)$: This follows from \cref{thm:Vossieck} and Theorem \ref{thm:der-dis-gentle}.

    $(1)\To(2)$:
    Assume that $A$ is derived-discrete. 
    By Theorem \ref{thm:der-dis-gentle}, it suffices to show that $A$ is graded Morita equivalent to a piecewise hereditary algebra of Dynkin type, or a graded gentle one-cycle algebra. By replacing $A=kQ/I$ with a graded Morita equivalent algebra (via a shift of the indecomposable projectives), we may assume that the degree zero part $\widetilde{Q}^{[0,0]}$ is connected.
    
    If $A = A_0$, the claim follows from \cref{thm:Vossieck}. Henceforth, assume $A \neq A_0$, so the graded quiver $Q$ contains at least one negative arrow.
    By Proposition \ref{prop:DD-DD} and \cref{thm:Vossieck}, every connected component of $\widetilde{A}^{[-n,0]}=k\wt{Q}^{[-n,0]}/\wt{I}^{[-n,0]}$ is either a piecewise hereditary algebra of Dynkin type or a gentle one-cycle algebra not satisfying the clock condition.
    
    \medskip
    \noindent\textit{Step 1: $Q$ has exactly one negative arrow.}
    
    Suppose $Q$ contains more than one negative arrow.
    Then for all sufficiently large $n$, the connected component of $\widetilde{Q}^{[-n,0]}$ containing $\widetilde{Q}^{[0,0]}$ has more than one cycle, contradicting the fact that each connected component of $\widetilde{A}^{[-n,0]}$ is a piecewise hereditary of Dynkin type or a gentle one-cycle algebra. Hence, there exists a subquiver $Q' (=\wt{Q}^{[0,0]})$ concentrated in degree $0$ and a single negative arrow $\alpha$ of degree $-k$ such that $Q = Q' \sqcup \{\alpha\}$.
    
    \medskip
    \noindent\textit{Step 2: Case analysis.}
    
    For each $n \geq 0$, let $\wt{A}_{\mathrm{main}}^{[-n,0]}$ denote the connected component of $\widetilde{A}^{[-n,0]}$ containing $A_0$. Note that the number of vertices of $\wt{A}_{\mathrm{main}}^{[-n,0]}$
    tends to infinity as $n \to \infty$; in particular, $\wt{A}_{\mathrm{main}}^{[-n,0]}$ cannot be of Dynkin type $E$ for all sufficiently large $n$. We thus consider the following three cases for $n \gg 0$.
    
    \medskip
    \noindent\textbf{Case 1:}
    $\wt{A}_{\mathrm{main}}^{[-n,0]}$ is a piecewise hereditary algebra of Dynkin type $D$ for every $n \gg 0$.
    
    \noindent
    The degree truncation of $A$ yields a semi-orthogonal decomposition
    \[
      \D^b(\grmod_0^{(-2nk,0]}\,A)
      =
      \mathcal{D}^b(\grmod_0^{(-2nk,-nk]}\,A)
      \perp
      \mathcal{D}^b(\grmod_0^{(-nk,0]}\,A).
    \]
    Restricting to the component $\wt{A}_{\mathrm{main}}^{[-2nk,0]}$, this induces a semi-orthogonal decomposition
    \[
      \D^b(\mod\wt{A}_{\mathrm{main}}^{(-2nk,0]}) = \T_2 \bot \T_1,
    \]
    where $\T_1$ and $\T_2$ are each equivalent to $\D^b(\mod \wt{A}_{\mathrm{main}}^{(-nk,0]})$, which is of type $D$. This contradicts Theorem~\ref{thm:SOD}.
    
    \medskip
    \noindent\textbf{Case 2:} $\wt{A}_{\mathrm{main}}^{[-n,0]}$ is a gentle tree algebra for every $n \gg 0$.

    \noindent Since $\wt{Q}^{[-n,0]}$ is a tree, the subquiver $Q'$ must also be a tree. Since $\wt{A}_{\mathrm{main}}^{[-n,0]}$ is a gentle tree algebra, the ideal $\wt{I}^{[-n,0]}$ is monomial and generated by paths of length two satisfying the gentle conditions. By the construction of $\wt{I}^{[-n,0]}$ from $I$, the ideal $I$ itself satisfies the gentle conditions. Since $Q'$ is a tree, $A$ is a graded gentle one-cycle algebra.
    
    \medskip
    \noindent\textbf{Case 3:}
    $\wt{A}_{\mathrm{main}}^{[-n,0]}$ is a gentle one-cycle algebra not satisfying the clock condition for some $n\gg0$.
    
    \noindent
    In this case, $Q'$ contains a cycle. Then, the quiver $\widetilde{Q}^{[-k,0]}$ contains two copies of $Q'$, and so contains two cycles. This is a contradiction.
\end{proof}

\bibliographystyle{alpha}
\bibliography{my.bib}

@article {KY18,
    AUTHOR = {Kalck, Martin and Yang, Dong},
     TITLE = {Derived categories of graded gentle one-cycle algebras},
   JOURNAL = {J. Pure Appl. Algebra},
  FJOURNAL = {Journal of Pure and Applied Algebra},
    VOLUME = {222},
      YEAR = {2018},
    NUMBER = {10},
     PAGES = {3005--3035},
      ISSN = {0022-4049,1873-1376},
   MRCLASS = {16E35 (16E45 18E30)},
  MRNUMBER = {3795632},
MRREVIEWER = {George\ Ciprian\ Modoi},
       DOI = {10.1016/j.jpaa.2017.11.011},
       URL = {https://doi.org/10.1016/j.jpaa.2017.11.011},
}

@article{AH24, title={WHEN IS THE SILTING-DISCRETENESS INHERITED?}, volume={256}, DOI={10.1017/nmj.2024.8}, journal={Nagoya Mathematical Journal}, author={Aihara, Takuma and Honma, Takahiro}, year={2024}, pages={905–927}}

@article{AI12,
  title={Silting mutation in triangulated categories},
  author={Aihara, Takuma and Iyama, Osamu},
  journal={Journal of the London Mathematical Society},
  volume={85},
  number={3},
  pages={633--668},
  year={2012},
  publisher={Wiley Online Library}
}

@article{AMY19,
  AUTHOR = {Adachi, Takahide and Mizuno, Yuya and Yang, Dong},
  TITLE = {Discreteness of silting objects and {$t$}-structures in triangulated categories},
  JOURNAL = {Proc. Lond. Math. Soc. (3)},
  VOLUME = {118},
  YEAR = {2019},
  NUMBER = {1},
  PAGES = {1--42},
  DOI = {10.1112/plms.12176},
}

@article {BPP18,
    AUTHOR = {Broomhead, Nathan and Pauksztello, David and Ploog, David},
     TITLE = {Discrete triangulated categories},
   JOURNAL = {Bull. Lond. Math. Soc.},
  FJOURNAL = {Bulletin of the London Mathematical Society},
    VOLUME = {50},
      YEAR = {2018},
    NUMBER = {1},
     PAGES = {174--188},
      ISSN = {0024-6093,1469-2120},
   MRCLASS = {18E30 (16E35 16G10)},
  MRNUMBER = {3778555},
MRREVIEWER = {Qunhua\ Liu},
       DOI = {10.1112/blms.12125},
       URL = {https://doi.org/10.1112/blms.12125},
}

@article{CSPP22,
  title={Functorially finite hearts, simple-minded systems in negative cluster categories, and noncrossing partitions},
  author={Coelho Sim{\~o}es, Raquel and Pauksztello, David and Ploog, David},
  journal={Compositio Mathematica},
  volume={158},
  number={1},
  pages={211--243},
  year={2022},
  publisher={London Mathematical Society}
}

@article {K94,
    AUTHOR = {Keller, Bernhard},
     TITLE = {Deriving {DG} categories},
   JOURNAL = {Ann. Sci. \'Ecole Norm. Sup. (4)},
  FJOURNAL = {Annales Scientifiques de l'\'Ecole Normale Sup\'erieure.
              Quatri\`eme S\'erie},
    VOLUME = {27},
      YEAR = {1994},
    NUMBER = {1},
     PAGES = {63--102},
      ISSN = {0012-9593},
   MRCLASS = {18E30 (16D90)},
  MRNUMBER = {1258406},
MRREVIEWER = {Jeremy\ Rickard},
       URL = {http://www.numdam.org/item?id=ASENS_1994_4_27_1_63_0},
}

@article {Vos01,
    AUTHOR = {Vossieck, Dieter},
     TITLE = {The algebras with discrete derived category},
   JOURNAL = {J. Algebra},
  FJOURNAL = {Journal of Algebra},
    VOLUME = {243},
      YEAR = {2001},
    NUMBER = {1},
     PAGES = {168--176},
      ISSN = {0021-8693,1090-266X},
   MRCLASS = {16G10 (16E10 18E30)},
  MRNUMBER = {1851659},
MRREVIEWER = {Bin\ Zhu},
       DOI = {10.1006/jabr.2001.8783},
       URL = {https://doi.org/10.1006/jabr.2001.8783},
}

@article {AR04,
    AUTHOR = {Athanasiadis, Christos A. and Reiner, Victor},
     TITLE = {Noncrossing partitions for the group {$D_n$}},
   JOURNAL = {SIAM J. Discrete Math.},
  FJOURNAL = {SIAM Journal on Discrete Mathematics},
    VOLUME = {18},
      YEAR = {2004},
    NUMBER = {2},
     PAGES = {397--417},
      ISSN = {0895-4801,1095-7146},
   MRCLASS = {06A07 (05A18 05E15)},
  MRNUMBER = {2112514},
MRREVIEWER = {Reinhard\ O. W. Franz},
       DOI = {10.1137/S0895480103432192},
       URL = {https://doi.org/10.1137/S0895480103432192},
}

@article {IT09,
    AUTHOR = {Ingalls, Colin and Thomas, Hugh},
     TITLE = {Noncrossing partitions and representations of quivers},
   JOURNAL = {Compos. Math.},
  FJOURNAL = {Compositio Mathematica},
    VOLUME = {145},
      YEAR = {2009},
    NUMBER = {6},
     PAGES = {1533--1562},
      ISSN = {0010-437X,1570-5846},
   MRCLASS = {16G20 (05E10 05E15 13F60)},
  MRNUMBER = {2575093},
MRREVIEWER = {Gregoire\ Dupont},
       DOI = {10.1112/S0010437X09004023},
       URL = {https://doi.org/10.1112/S0010437X09004023},
}

@article {H51,
    AUTHOR = {Hashimoto, Junji},
     TITLE = {On direct product decomposition of partially ordered sets},
   JOURNAL = {Ann. of Math. (2)},
  FJOURNAL = {Annals of Mathematics. Second Series},
    VOLUME = {54},
      YEAR = {1951},
     PAGES = {315--318},
      ISSN = {0003-486X},
   MRCLASS = {09.1X},
  MRNUMBER = {43067},
MRREVIEWER = {G.\ Birkhoff},
       DOI = {10.2307/1969532},
       URL = {https://doi.org/10.2307/1969532},
}

@article{LZ21,
title={The geometric model of gentle one-cycle algebras},
author={Liu, Yu-Zhe and Zhang, Chao},
journal={Bull. Malays. Math. Sci. Soc.},
volume={44},
number={4},
pages={2489--2505},
year={2021}
}

@article{OPS18,
title={A geometric model for the derived category of gentle algebras},
author={Opper, Sebastian and Plamondon, Pierre-Guy and Schroll, Sibylle},
journal={arXiv preprint arXiv:1801.09659},
year={2018}
}

@article{Sch15,
title={Trivial extensions of gentle algebras and {B}rauer graph algebras},
author={Schroll, Sibylle},
journal={J. Algebra},
volume={444},
pages={183--200},
year={2015}
}

@article{ABCP10,
title={Gentle algebras arising from surface triangulations},
author={Assem, Ibrahim and Br{\"u}stle, Thomas and Charbonneau-Jodoin, Gabrielle and Plamondon, Pierre-Guy},
journal={Algebra Number Theory},
volume={4},
number={2},
pages={201--229},
year={2010}
}

@article{LGH24,
  title={Homological dimensions of gentle algebras via geometric models},
  author={Liu, Yu-Zhe and Gao, Hanpeng and Huang, Zhaoyong},
  journal={Sci. China Math.},
  volume={67},
  number={4},
  pages={733--766},
  year={2024}
}

@article{LP20,
title={Derived equivalences of gentle algebras via Fukaya categories},
author={Lekili, Yank{\i} and Polishchuk, Alexander},
journal={Math. Ann.},
volume={376},
number={1},
pages={187--225},
year={2020}
}

@article{HKK17,
title={Flat surfaces and stability structures},
author={Haiden, Fabian and Katzarkov, Ludmil and Kontsevich, Maxim},
journal={Publ. Math. Inst. Hautes \'{E}tudes Sci.},
volume={126},
pages={247--318},
year={2017}
}

@article{JSW23,
title={A complete derived invariant and silting theory for graded gentle algebras},
author={Jin, Haibo and Schroll, Sibylle and Wang, Zhengfang},
journal={arXiv preprint arXiv:2303.17474},
year={2023}
}

@article{CJSW25,
title={On the $\tau$-tilting finiteness and silting-discreteness of graded (skew-) gentle algebras},
author={Chang, Wen and Jin, Haibo and Schroll, Sibylle and Wang, Qi},
journal={arXiv preprint arXiv:2512.24316},
year={2025}
}

@article{OZ22,
title={Derived equivalence classification of {B}rauer graph algebras},
author={Opper, Sebastian and Zvonareva, Alexandra},
journal={Adv. Math.},
volume={402},
pages={108341},
year={2022}
}

@article{Fus26,
title={Silting-discrete graded path algebras},
author={Fushimi, Riku},
journal={arXiv preprint arXiv:2605.23704},
year={2026}
}

@article{AS87,
  title={Iterated tilted algebras of type $\widetilde{A}_{n}$},
  author={Assem, Ibrahim and Skowroński, Andrzej},
  journal={Math. Z.},
  volume={195},
  number={2},
  pages={269--290},
  year={1987}
}

@article{BS21,
  title={A geometric model for the module category of a gentle algebra},
  author={Baur, Karin and Coelho Simões, Raquel},
  journal={Int. Math. Res. Not. IMRN},
  volume={2021},
  number={15},
  pages={11357--11392},
  year={2021}
}

@article{BGS04,
  title={Classification of discrete derived categories},
  author={Bobi\'{n}ski, Grzegorz and Gei{\ss}, Christof and Skowro\'{n}ski, Andrzej},
  journal={Open Math.},
  volume={2},
  number={1},
  pages={19--49},
  year={2004}
}

@article{AM17,
  title={Classifying tilting complexes over preprojective algebras of Dynkin type},
  author={Aihara, Takuma and Mizuno, Yuya},
  journal={Algebra Number Theory},
  volume={11},
  number={6},
  pages={1287--1315},
  year={2017}
}

@article{Aih13,
  title={Tilting-connected symmetric algebras},
  author={Aihara, Takuma},
  journal={Algebr. Represent. Theory},
  volume={16},
  number={3},
  pages={873--894},
  year={2013}
}

@article{Opp25,
title={Autoequivalences of Fukaya categories of surfaces and graded gentle algebras},
author={Opper, Sebastian},
journal={arXiv preprint arXiv:2510.11543},
year={2025}
}

@article{YY21,
title={The equivalence of two notions of discreteness of triangulated categories},
author={Yao, Lingling and Yang, Dong},
journal={Algebr. Represent. Theory},
volume={24},
number={5},
pages={1295--1312},
year={2021}
}
\Addresses

\end{document}